\newif\ifArxiv
\Arxivtrue

\ifArxiv	
\documentclass[11pt,a4paper]{article}
\usepackage{authblk}
\usepackage{cite}

\else		
\documentclass{gNST2e}
\usepackage{subfigure}
\fi

\usepackage{graphicx}
\usepackage{epstopdf} 
\usepackage{mleftright}
\usepackage{latexsym, amsmath, amsthm, amssymb}
\usepackage[usenames,dvipsnames]{color}
\usepackage{verbatim}
\usepackage{booktabs}
\usepackage{enumitem}
\usepackage{array}
\usepackage{slashbox}
\usepackage[T1]{fontenc}
\newcolumntype{L}[1]{>{\raggedright\let\newline\\\arraybackslash\hspace{0pt}}m{#1}}
\newcolumntype{C}[1]{>{\centering\let\newline\\\arraybackslash\hspace{0pt}}m{#1}}

\usepackage{algorithm}
\usepackage[]{algpseudocode}
\newcommand*\Let[2]{\State #1 $\gets$ #2}
\newcommand*\LCall[2]{#1$\Big($#2$\Big)$}

\DeclareMathOperator*{\argmax}{arg\,max}

\def\OP{{\rm OP}}
\def\CE{{\rm CE}}
\def\eCE{{\rm eCE}}
\def\CEofOP{{\rm CEofOP}}
\def\AR{{\rm AR}}
\def\NL{{\rm NL}}

\def\Likelihood{{\rm Lkl}}

\def\RMSE{\mathrm{RMSE}}

\theoremstyle{plain}
\newtheorem{theorem}{Theorem}
\newtheorem{lemma}[theorem]{Lemma}
\newtheorem{corollary}[theorem]{Corollary}
\newtheorem{proposition}[theorem]{Proposition}

\theoremstyle{definition}
\newtheorem{definition}{Definition}
\newtheorem{example}{Example}
\newtheorem*{remark}{Remark}

\newtheoremstyle{myExperiment}
{12pt}
{12pt}
{}
{}
{\bfseries}
{:}
{.5em}
{}

\theoremstyle{myExperiment}

{
\begin{list}{}%
{
\setlength{\leftmargin}{#1}}%
\item[]%
}
{\end{list}
}

\ifArxiv
\title{Change-point detection using the conditional entropy of ordinal patterns}
\author[1,2]{Anton M.~Unakafov\thanks{Corresponding address: Institute of Mathematics, University of L\"ubeck, Ratzeburger Allee 160, Building 64, 23562 L\"ubeck, Germany. e-mail: anton@math.uni-luebeck.de (A.M. Unakafov)}}
\author[1]{Karsten Keller}
\affil[1]{Institute of Mathematics, University of L\"ubeck}
\affil[2]{Graduate School for Computing in Medicine and Life Sciences, University of L\"ubeck}
\begin{document}
\maketitle

\begin{abstract}
\noindent This paper is devoted to change-point detection using only the ordinal structure of a time series.
A statistic based on the conditional entropy of ordinal patterns characterizing the local up and down in a time series is introduced and investigated.
The statistic requires only minimal a priori information on given data and shows good performance in numerical experiments.

\end{abstract}

\else
\begin{document}
\title{Change-point detection using the conditional entropy of ordinal patterns}

\author{
\name{Anton M.~Unakafov\textsuperscript{a,b}$^{\ast}$\thanks{$^\ast$Corresponding author. Email: anton@math.uni-luebeck.de}
and Karsten Keller\textsuperscript{a}}
\affil{\textsuperscript{a}Institute of Mathematics, University of L\"ubeck, D-23562 L\"ubeck, Germany;
\textsuperscript{b}Graduate School for Computing in Medicine and Life Sciences, University of L\"ubeck, D-23562 L\"ubeck, Germany}
\received{December 2016}
}

\maketitle

\begin{abstract}
\noindent This paper is devoted to change-point detection using only the ordinal structure of a time series.
A statistic based on the conditional entropy of ordinal patterns characterizing the local up and down in a time series is introduced and investigated.
The statistic requires only minimal a priori information on given data and shows good performance in numerical experiments.
\end{abstract}


\begin{classcode}62G99; 60G99; 62M10\end{classcode}
\fi

\section{Introduction}\label{intro}

Most of real-world time series are non-stationary, that is some of their properties change over time.
A model for some non-stationary time series is provided by a piecewise stationary stochastic process:
its properties are locally constant except for certain time-points called {\it change-points}, where some properties change abruptly \cite{BassevilleNikiforov1993}.
Detecting change-points is a classical problem and there are many methods for tackling it
\cite{BassevilleNikiforov1993,BrodskyDarkhovsky1993,CarlsteinMullerSiegmund1994, BrodskyDarkhovsky2000, LavielleTeyssiere2007}.
However, most of the existing methods have a common drawback:
they require certain a priori information about the time series.
It is necessary to know either a family of stochastic processes providing a model for the time series
(see for instance \cite{Davis2006} where AR processes are considered)
or at least to know which characteristics (mean, standard deviation, etc.) of the time series reflect the change
(see \cite{BrodskyDarkhovsky2000, Preuss2015}).
In real-world applications such information is often unavailable \cite{BrodskyDarkhovskyKaplanShishkin1999}.

Here we suggest a new method for change-point detection that requires minimal a priori knowledge:
we only assume that the changes affect the evolution rule linking the past of the process with its future
(a formal description of the considered processes is provided by Definition~\ref{StructureChanges_def}).
A natural example of such change is an alteration of the increments distribution.

Our method is based on ordinal pattern analysis, a promising approach to real-valued time series analysis
\cite{BandtPompe2002, BandtShiha2007, KellerSinnEmonds2007, Amigo2010, PompeRunge2011, UnakafovaKeller2013, KellerUnakafovUnakafova2014}.
In ordinal pattern analysis one considers order relations between values of a time series instead of the values themselves.
These order relations are coded by ordinal patterns; specifically,
an ordinal pattern of an order $d \in \mathbb{N}$ describes order relations between $(d+1)$ successive points of a time series.
The main step of ordinal pattern analysis is the transformation of an original time series into a sequence of ordinal patterns.
A result of this transformation is demonstrated in Figure~\ref{MotivatingExample_fig} for order $d = 1$.

\begin{figure}[!th]
\centering
\includegraphics[scale=0.136]{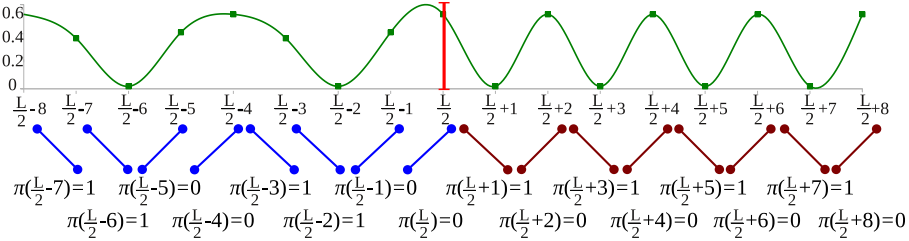}
\caption{A part of a piecewise stationary time series with a change-point at $t=L/2$ (marked by a vertical line) and corresponding ordinal patterns of order $d=1$ (below the plot)}
\label{MotivatingExample_fig}
\end{figure}

For detecting a change-point $t^\ast \in \mathbb{N}$ in a time series $x = \big(x(t)\big)_{t=0}^L$ with values in ${\mathbb R}$, one generally considers $x$ as a realization of a stochastic process $X$ and computes for $x$
a statistic $S(t; x)$ that should reach its maximum at $t = t^\ast$.
Here we suggest a statistic on the basis of the conditional entropy of ordinal patterns introduced in \cite{UnakafovKeller2014}.

Let us provide an `obvious' example only to motivate our approach and to illustrate its idea.
\begin{example}\label{CEofOPToyEx}
Consider a time series $\big(x(t)\big)_{t=0}^L$, its central part is shown in Figure~\ref{MotivatingExample_fig}.
The time series is periodic before and after $L/2$, but at $L/2$ there occurs a change (marked by a vertical line): the ``oscillations'' become faster.
Figure~\ref{MotivatingExample_fig} also presents the ordinal patterns $\pi (t)$ of order $d=1$ at times $t$ underlying the time series.
Note that there are only two ordinal patterns of order $1$: the increasing (coded by $0$) and the decreasing (coded by $1$).
Both ordinal patterns occur with the same frequency before and after the change-point.

However, the transitions between successive ordinal patterns are changed at $\frac{L}{2}$.
Indeed, before the change-point $L/2$ both ordinal patterns have two possible successors
(for instance, the ordinal pattern $\pi(L/2-5) = 0$ is succeeded by the ordinal pattern $\pi(L/2-4) = 0$, which in turn is succeeded by the ordinal pattern $\pi(L/2-3) = 1$),
whereas after the change-point the ordinal patterns $0$ and $1$ are alternating.
A measure of diversity of transitions between ordinal patterns is provided by the conditional entropy of them.
For the sequence $\big(\pi(k)\big)_{k=1}^L$ of ordinal patterns of order $1$ the (empirical) conditional entropy for $t = 2,3,\ldots,L$ is defined as follows:
\begin{align*}
\eCE\Big(\big(\pi(k)\big)_{k=1}^t\Big) = &-\sum_{i=0}^1\sum_{j=0}^1 \frac{n_{i,j}(t)}{t-1} \ln \frac{n_{i,j}(t)}{n_{i}(t)}\\
\text{with } n_{i,j}(t) &= \#\{l = 1,2,\ldots,t-1 \mid \pi(l) = i, \pi(l+1) = j \},\\
n_{i}(t)   &= \#\{l = 1,2,\ldots,t-1 \mid \pi(l) = i\}
\end{align*}
(throughout the paper, $0\ln 0 :=0$ and $0/0 :=0$, and $\# A$ denotes the number of elements of a set $A$).

To detect change-points we use a test statistic for $d = 1$ defined as follows:
\begin{align*}
\CEofOP(\theta L) =& (L - 2)\,\eCE\Big(\big(\pi(k)\big)_{k=1}^L\Big)\\
&- (\theta L - 1)\,\eCE\Big(\big(\pi(k)\big)_{k=1}^{\theta L}\Big)\\
&- \big(L -(\theta L + 1) \big)\,\eCE\Big(\big(\pi(k)\big)_{k=\theta L+1}^{L}\Big),
\end{align*}
for $\theta \in (0,1)$ with $\theta L \in \mathbb{N}$.
According to the properties of conditional entropy (see Section~\ref{SecCEofOPstat} for details), $\CEofOP(\theta L)$ attains its maximum when $\theta L$ coincides with a change-point.
Figure~\ref{CeofOPstat_ToyExample} demonstrates this for the time series from Figure~\ref{MotivatingExample_fig}.

\begin{figure}[!th]
\centering
\includegraphics[scale=0.4]{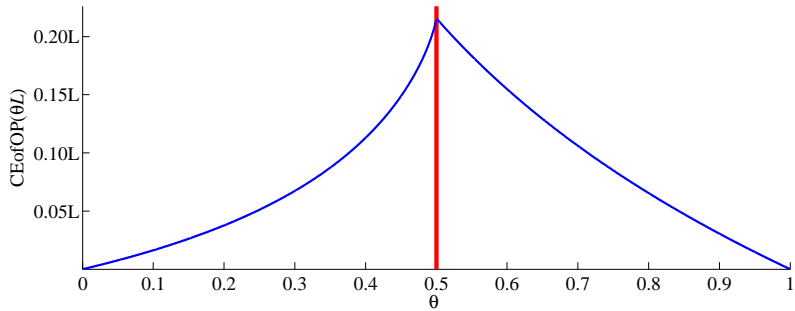}
\caption{Statistic $\CEofOP(\theta L)$ for the sequence of ordinal patterns of order $1$ for the time series from Figure~\ref{MotivatingExample_fig}}
\label{CeofOPstat_ToyExample}
\end{figure}
\end{example}

For simplicity and in view of real applications, in Example~\ref{CEofOPToyEx}  we define ordinal patterns and the $\CEofOP$ statistic immediately for concrete time series.
However, for theoretical consideration it is necessary to define the $\CEofOP$ statistic for stochastic processes. For this we refer to Section~\ref{SecCEofOPstat}.

To illustrate applicability of the $\CEofOP$ statistic let us discuss a real-world data example:

\begin{example}\label{CEofOPeegEx}
Here we consider EEG recording 14 from the sleep EEG dataset kindly provided by Vasil Kolev
(details and other results for this dataset are provided in \cite[Subsection~5.3.2]{Unakafov2015}).
We employ the following procedure for an automatic discrimination between sleep stages from the EEG time series:
first we split time series into pseudo-stationary intervals by finding change-points with the $\CEofOP$ statistic (change-points are detected in each EEG channel separately), then we cluster all the obtained intervals.
Figure~\ref{figureSleepEEG_OrdinalDiscrim} illustrates the outcome of the proposed discrimination for single EEG channel in comparison with the manual scoring by an expert;
the automated identification of a sleep type (waking, REM, light sleep, deep sleep) is correct for 79.6\% of 30-second epochs. Note that the borders of the segments (that is the detected change-points) in most cases correspond to the changes of sleep stage.
    \begin{figure}[!th]
      \hspace*{-2mm}
      \includegraphics[scale=0.52]{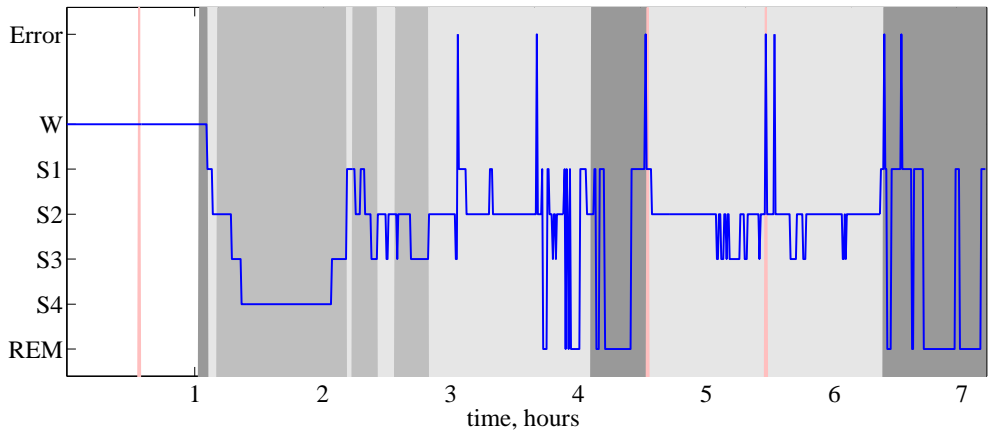}
      \caption{Hypnogram (bold curve) and the results of ordinal-patterns-based discrimination of sleep EEG. 
      		  Here ordinate axis represents the results of the expert classification: W stands for waking, stages S1, S2 and S3, S4 indicate light and deep sleep, respectively, REM stands for REM sleep and Error -- for unclassified samples. Results of ordinal-patterns-based discrimination are represented by the background colour:  	
              white colour indicates epochs classified as waking state, light gray -- as light sleep,
               gray -- as deep sleep, dark gray -- as REM, red colour indicates unclassified segments}
      \label{figureSleepEEG_OrdinalDiscrim}
    \end{figure}
\end{example}

The $\CEofOP$ statistic was first introduced in \cite{KellerUnakafovUnakafova2014},
where we have employed it as a component of a method for sleep EEG discrimination.
However no theoretical details of the method for change-point detection were provided there.
This paper aims to fill in this gap and provides a justification of the $\CEofOP$ statistic.

This paper is organized as follows.
In Section~\ref{SecPrelim} we provide a brief introduction into ordinal pattern analysis.
In particular, we define the conditional entropy of ordinal patterns and discuss its properties.
In Section~\ref{SecCEofOPstat} we investigate the properties of the $\CEofOP$ statistic.
We also suggest there a method for detecting multiple change-points via the $\CEofOP$ statistic.
Section~\ref{SecExperiments} is devoted to a comparison of this method with two other (classical and ordinal-patterns-based) methods for change-point detection
by performing experiments on realizations of piecewise stationary stochastic processes.
In Section~\ref{SecConcl} we summarize the results and state open problems.
Finally, in supplementary Section~\ref{sec_stationary} we investigate the asymptotic properties of the $\CEofOP$ statistic.

\section{Preliminaries}\label{SecPrelim}

Central objects of the following are stochastic processes $X = \big(X(t)\big)_{t=n}^m$ on a probability space $(\Omega, \mathcal{A}, \mathbb{P})$ with values in ${\mathbb R}$.
Here $n\in {\mathbb N}_0$ and $m\in {\mathbb N}\cup\{\infty\}$ allowing both finite and infinite lengths of processes. We consider only univariate stochastic processes to keep notation simple, though there are no principal restrictions on the dimension of a process. $X = \big(X(t)\big)_{t=n}^m$ is \emph{stationary} if for all $t_1,t_2,\ldots ,t_k,s$
with $t_1,t_2,\ldots ,t_k,t_1+s,t_2+s,\ldots ,t_k+s\in\{n,n+1,\ldots ,m\}$ the distributions of $(X_{t_i})_{i=1}^k$ and $(X_{t_i+s})_{i=1}^k$ coincide.

Throughout this paper we discuss detection of change-points in a piecewise stationary stochastic process.
Simply speaking, a piecewise stationary stochastic process is obtained by ``gluing'' several pieces of stationary stochastic processes
(for a formal definition of piecewise stationarity see, for instance, \cite[Section~3.1]{Stoffer2012}).

In this section we recall the basic facts from ordinal pattern analysis (Subsection~\ref{subsec_OPprocess}),
present the idea of ordinal-patterns-based change-point detection (Subsection~\ref{subsec_CPinOPdetect}),
and define the conditional entropy of ordinal patterns (Subsection~\ref{subsec_CE}).

\subsection{Ordinal patterns}\label{subsec_OPprocess}

Let us recall the definition of an ordinal pattern \cite{KellerSinnEmonds2007,UnakafovaKeller2013, KellerUnakafovUnakafova2014}.
For $d \in \mathbb{N}$ denote the set of permutations of $\lbrace 0, 1, \ldots, d\rbrace$ by $S_d$.
\begin {definition}\label{OrdPatternDef}\index{ordinal pattern}
We say that a real vector $(x_0, x_1, \ldots, x_d)$ has {\it ordinal pattern} $\OP(x_0, x_1, \ldots, x_d) = (r_0, r_1,\ldots, r_d) \in S_d$ {\it of order} $d \in {\mathbb N}$ if
\begin {equation*}
x_{r_0} \geq x_{r_1} \geq \ldots \geq x_{r_d}
\end {equation*}
and
\begin {equation*}
r_{l-1} > r_{l} \text{ for } x_{r_{l-1}} = x_{r_{l}}.
\end {equation*}
\end {definition}
As one can see there are $(d+1)!$ different ordinal patterns of order $d$.

\begin{definition}\label{OPseq}
Given a stochastic process $X = \big(X(t)\big)_{t=0}^L$ for $L\in {\mathbb N}\cup\{\infty\}$,
the sequence $\Pi^d = \big(\Pi(t)\big)_{t=d}^L$ with
\begin{equation*}
\Pi(t) = \OP\big(X(t-d), X(t-d+1), \ldots, X(t)\big)
\end{equation*}
is called the {\it abstract sequence of ordinal patterns} of order $d \in {\mathbb N}$ of the process $X$.
Similarly, given $x=(x(t))_{t=0}^L$ a realization of $X$, the {\it sequence of ordinal patterns} of order $d$ of $x$ is defined as
$\pi^{d,L} = \big(\pi(t)\big)_{t=d}^L$ with
\begin{equation*}
\pi(t) = \OP\big(x(t-d), x(t-d+1), \ldots, x(t)\big).
\end{equation*}
For simplicity, we say that $L$ in the case $L\in {\mathbb N}$ is the length of the sequence $\pi^{d,L}$ though, in fact, it consists of $(L-d+1)$ elements.
\end{definition}
\begin{definition}\label{OrdinalStationarity}
A stochastic process $X = \big(X(t)\big)_{t=0}^L$ for $L\in {\mathbb N}\cup\{\infty\}$ is said to be \emph{ordinal-$d$-stationary} if for all
$i \in S_d$ the probability $\mathbb{P}\big(\Pi(t) = i\big)$ does not depend on $t$ for $d\leq t\leq L$.
In this case we call
\begin{equation}\label{OPprob}
p_i = \mathbb{P}\big(\Pi(t) = i\big)
\end{equation}
the {\it probability of the ordinal pattern} $i \in S_d$ in $X$.
\end{definition}
The idea of ordinal pattern analysis is to consider the sequence of ordinal patterns and the ordinal patterns distribution obtained from it instead of the original time series.
Though implying the loss of all the metric information, this often allows to extract some relevant information from a time series, in particular, when it comes from a complex system.
For example, ordinal pattern analysis provides estimators of the Kolmogorov-Sinai entropy \cite{BandtKellerPompe2002, Keller2012, UnakafovKeller2014} of dynamical systems,
measures of time series complexity \cite{BandtPompe2002, Pompe2013, KellerUnakafovUnakafova2014},
measures of coupling between time series \cite{PompeRunge2011, HarunaNakajima2013}
and estimators of parameters of stochastic processes \cite{BandtShiha2007, SinnKeller2011},
see also \cite{Amigo2010, AmigoKeller2013} for a review of applications to real-world time series.
Methods of ordinal pattern analysis are invariant with respect to strictly-monotone distortions of time series \cite{KellerSinnEmonds2007}, do not need information about range of measurements, and are computationally simple 
\cite{UnakafovaKeller2013}. This qualifies it for application in the case that no much is known about the system behind a time series, possibly as a first exploration step.

For a discussion of the properties of a sequence of ordinal patterns we refer to \cite{BandtShiha2007, SinnKeller2008, SinnKeller2011, Bandt2014, ElizaldeMartinez2014}.
For the following we need two results stated below.
\begin{lemma}[Corollary~2 from \cite{SinnKeller2008}]\label{OPseqIsStat}
Each process $X = \big(X(t)\big)_{t\in {\mathbb N}_0}$ with associated stationary increment process $(X(t) - X(t-1))_{t\in \mathbb{N}}$ is ordinal-$d$-stationary for each $d\in {\mathbb N}$.
\end{lemma}
The next result is a direct consequence of \cite[Lemma~4]{UnakafovKeller2014} and generalizes the main result of \cite{Nagaraja1982}.
\begin{theorem}\label{OPseqIsMC}
If $X$ is a Markov chain with values in a two-element-set (contained in ${\mathbb R}$), then the corresponding abstract sequence $\Pi^d$ of ordinal patterns also forms a Markov chain for every order $d \in \mathbb{N}$.
\end{theorem}
In relation to this theorem, note that in general one does not need values in ${\mathbb R}$ but in a totally ordered set. It is shown in \cite{Nagaraja1982} and \cite[Subsection~3.3.3]{Unakafov2015} that if $X$ is a Markov chain in a finite set of more than two elements, $\Pi^d$ does not generally form a Markov chain.

Probability distributions of ordinal patterns are known only for some special cases of stochastic processes \cite{BandtShiha2007, SinnKeller2008, ElizaldeMartinez2014}.
In general one estimates probabilities of ordinal patterns by their empirical probabilities.
Consider a sequence $\pi^{d,L}$ of ordinal patterns.
For any $t \in \{d+1, d+2,\ldots, L\}$ the frequency of occurrence of an ordinal pattern $i \in S_d $ among the first $(t-d)$ ordinal patterns of the sequence is given by
\begin{equation}\label{OPfrequency}
   n_{i}(t) = \#\{l \in \{d, d+1, \ldots, t-1\} \mid \pi(l) = i \}.
\end{equation}
A natural estimator of the probability of an ordinal pattern $i$ in the ordinal-$d$-stationary case is provided by its relative frequency in the sequence $\pi^{d,L}$:
   \begin{equation*}
       \widehat{p}_i = \frac{n_{i}(L)}{L-d}.
  \end{equation*}

\subsection{Detecting change-points in the sequence of ordinal patterns}\label{subsec_CPinOPdetect}

Sequences of ordinal patterns are invariant to certain changes in the original stochastic process $X$,
such as shifts (adding a constant to the process) \cite[Subsection~3.4.3]{Amigo2010} and scaling (multiplying the process with a positive constant) \cite{KellerSinnEmonds2007}.
However, in many cases changes in the original process $X$ affect also the corresponding abstract sequences of ordinal patterns and ordinal patterns distributions.
On the one hand, this impedes application of ordinal pattern analysis to non-stationary time series
since most of ordinal-patterns-based quantities require ordinal-$d$-stationarity of a time series \cite{BandtPompe2002, Amigo2010, PompeRunge2011}
and may be unreliable when this condition fails.
On the other hand, one can detect change-points in the original process by detecting changes in the sequence of ordinal patterns.

First ideas of using ordinal patterns for detecting change-points were formulated in \cite{SinnGhodsiKeller2012, SinnKellerChen2013, Bandt2014, YuanWangHuangSha2015, SchnurrDehling2015}.
The advantage of the ordinal-patterns-based methods is that they require less information than most of the existing methods for change-point detection:
it is assumed neither that the stochastic process is from a specific family nor that the change affects specific characteristic of the process.
Instead, we consider further change-points with the following property.

\begin{definition}\label{StructureChanges_def}
Let $\big( X(t) \big)_{t=0}^L$ with $L\in {\mathbb N}\cup\{\infty\}$ be a piecewise stationary stochastic process with a change-point $t^\ast \in \mathbb{N}$.
We say that $t^\ast$ is a {\it ordinal change-point} if there exist some $m,n\in {\mathbb N}$ with $m<t^\ast<n\leq L$ and some $d\in {\mathbb N}$
such that $\big( X(t) \big)_{t=m}^{t^\ast}$ and $\big( X(t) \big)_{t^\ast+1}^n$ are ordinal-$d$-stationary but $\big( X(t) \big)_{t=m}^n$ is not.
\end{definition}
This approach seems to be natural for many stochastic processes and real-world time series.

\begin{remark}
Note that a change-point where the change in mean occurs, need not be ordinal,
since the mean is irrelevant for the distribution of ordinal patterns \cite[Subsection~3.4.3]{Amigo2010}.
However, there are many methods that effectively detect changes in mean;
the method proposed here is intended for use in a more complex case, when there is no classical method, or it is not clear, which of them to apply.
\end{remark}

We illustrate Definition~\ref{StructureChanges_def} by two examples.
Piecewise stationary autoregressive processes considered in Example~\ref{ExampleARwithChanges} are classical and provide models for linear time series.   
Since many real-world time series are non-linear, further we introduce in Example~\ref{ExampleNLwithChanges} a process originated from non-linear dynamical systems. 
These two types of processes are used throughout the paper for empirical investigation of change-point detection methods.

\begin{example}\label{ExampleARwithChanges}
A first order {\it piecewise stationary autoregressive $(\AR)$ process} with change-points $t^\ast_1, t^\ast_2, \ldots, t^\ast_{N_\text{st}-1}$ is defined as
\begin{equation*}
 \AR\big((\phi_1, \phi_2,\ldots, \phi_{N_\text{st}}), (t^\ast_1, t^\ast_2, \ldots, t^\ast_{N_\text{st}-1}) \big) = \big(\AR(t)\big)_{t=0}^L,
\end{equation*}
where $\phi_1, \phi_2,\ldots, \phi_{N_\text{st}} \in [0, 1)$ are the parameters of the autoregressive model and
\begin{equation*}
 \AR(t)	=  \phi_k \AR(t-1) + \epsilon(t),
\end{equation*}
for all $t \in \{t^\ast_{k-1}+1,t^\ast_{k-1}+2,\ldots,t^\ast_{k} \}$ for $k = 1,2,\ldots,N_\text{st}$, where $t^\ast_0:=0$ and $t^\ast_{N_\text{st}}:=L$,
with $\epsilon$ being the standard white Gaussian noise, and $\AR(0): = \epsilon(0)$.
$\AR$ processes are often used for the investigation of methods for change-points detection (see, for instance, \cite{SinnGhodsiKeller2012, SinnKellerChen2013}),
since they provide models for a wide range of real-world time series.
Figure~\ref{OPdistrLogistAR}a illustrates a realization of a `two piece' AR process with a change-point at $L/2$.
By \cite[Proposition~5.3]{BandtShiha2007} the distributions of ordinal patterns of order $d \geq 2$ reflect change-points for piecewise stationary AR processes.
Figure~\ref{OPdistrLogistAR}c illustrates this for the realization from Figure~\ref{OPdistrLogistAR}a:
empirical probability distributions of ordinal patterns of order $d = 2$ before and after the change-point $L/2$ differ significantly.
\end{example}

\begin{example}\label{ExampleNLwithChanges}
A classical example of non-linear system is provided by a logistic map on the unit interval:   
\begin {equation}\label{LogisticMapEq}
      x(t) = r x(t-1) \big(1-x(t-1)\big),
\end {equation}
with $t \in \mathbb{N}$, for $x(0) \in [0,1]$ and $r \in [1, 4]$. 
The behaviour of this map significantly varies for different value $r$; 
we are especially interested in $r \in [3.57, 4]$ with chaotic behaviour.
In the latter situation there exists an invariant ergodic measure absolutely continuous with respect to the Lebesgue measure \cite{Thunberg2001, Lyubich2012},
therefore \eqref{LogisticMapEq} defines a stationary stochastic process $\NL_0$:
\begin{equation*}
	\NL_0(t) = r \big(1-\NL_0(t-1)\big) \NL_0(t-1),
\end{equation*}
with $\NL_0(0)\in [0,1]$ being a uniformly distributed random number. 
Note that for almost all $r \in [3.57, 4]$ (more generally, for all $r \in [0, 4]$), either the map $\NL_0$ is chaotic or hyperbolic roughly meaning that an attractive periodic orbit is dominating it. This is a deep result in one-dimensional dynamics (see \cite{Lyubich2012} for details). In the hyperbolic case after some transient behaviour numerically one only sees some periodic orbit, which has long periods in the interval $r \in [3.57, 4]$. From the practical viewpoint, i.e. when considering short orbits, dynamics for that interval can be considered as chaotic since already small changes of $r$ provide chaotic behaviour also in the theoretical sense.

Let us include some {\em observational noise} by adding standard white Gaussian noise $\epsilon$ to an orbit:
\begin{equation*}
	\NL(t) = \NL_0(t) + \sigma \epsilon(t),
\end{equation*}
where $\sigma >0$ is the level of noise.

Orbits of logistic maps, particularly with observational noise, are often used as a studying and illustrating tool of non-linear time series analysis (see \cite{LinzLuecke1986, Diks1999}).
This justifies as a natural object for study a {\it piecewise stationary noisy logistic $(\NL)$ process} with change-points $t^\ast_1, t^\ast_2, \ldots, t^\ast_{N_\text{st}-1}$, defined as
\begin{align*}
\NL\big((r_1, \ldots, r_{N_\text{st}}), (\sigma_1, \ldots, \sigma_{N_\text{st}}), (t^\ast_1, t^\ast_2, \ldots, t^\ast_{N_\text{st} - 1}) \big) = \big(\NL(t)\big)_{t=0}^L,
\end{align*}
where $r_1, \ldots, r_{N_\text{st}} \in [3.57, 4]$ are the values of control parameter, $\sigma_1, \ldots, \sigma_{N_\text{st}} > 0$ are the levels of noise, and
\begin{equation*}
\NL(t)	=  \NL_0(t) + \sigma_k \epsilon(t),
\end{equation*}
with
\begin{equation*}
\NL_0(t)	=  r_k \big(1-\NL_0(t-1)\big) \NL_0(t-1),
\end{equation*}
for all $t \in \{t^\ast_{k-1}+1,t^\ast_{k-1}+2,\ldots,t^\ast_{k}\}$ for $k = 1,2,\ldots,N_\text{st}$, with $t^\ast_0:=0$, $t^\ast_{N_\text{st}}:=L$ and $\NL_0(0) = \epsilon(0)$.

Figure~\ref{OPdistrLogistAR}b shows a realization of a `two-piece' NL process with a change-point at $L/2$;
as one can see in Figure~\ref{OPdistrLogistAR}d, the empirical distributions of ordinal patterns of order $d = 2$ before the change-point and after the change-point do not coincide.
In general, the distributions of ordinal patterns of order $d \geq 1$ reflect change-points for the NL processes (which can be easily checked).

\end{example}

\begin{figure}[h]
\begin{minipage}[h]{0.49\hsize}
\centering
\includegraphics[scale=0.50]{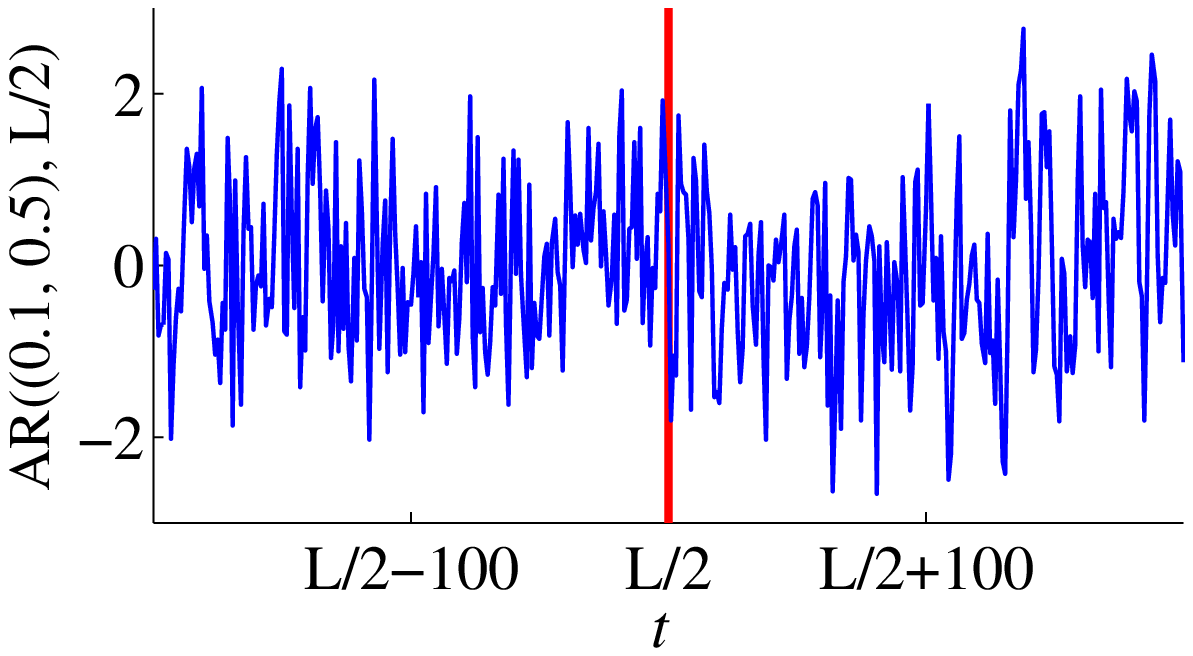}
	
\hspace{9mm}(a)

\includegraphics[scale=0.47]{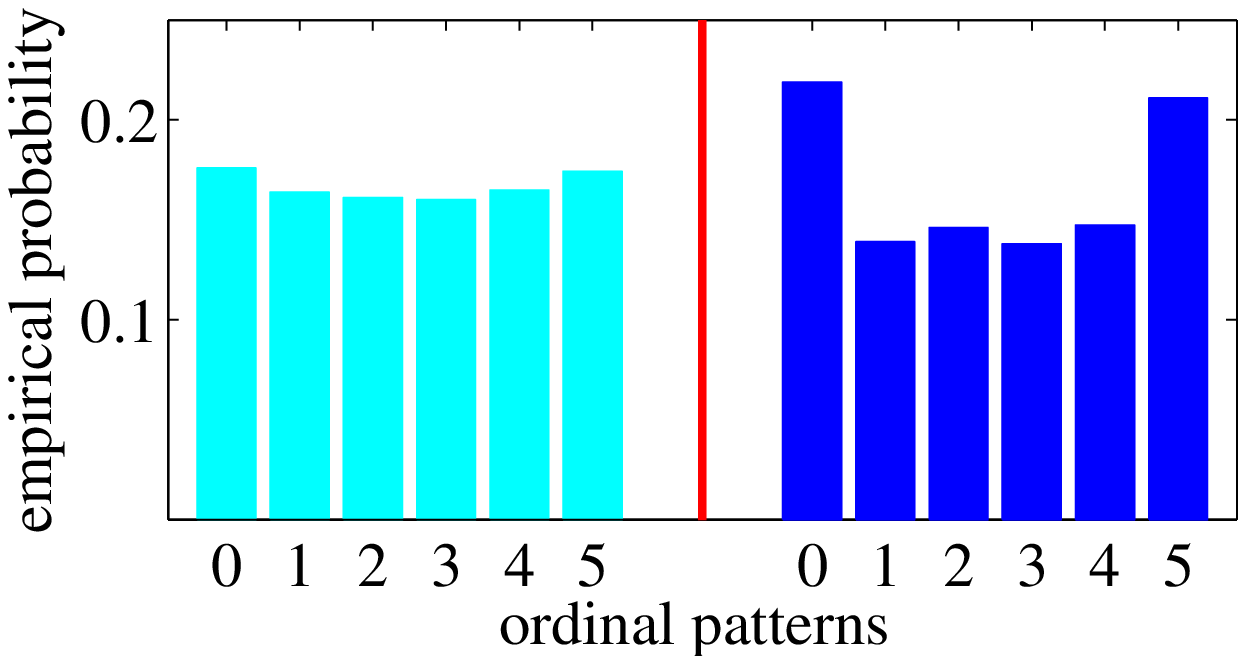}
	
\hspace{9mm}(c)

\end{minipage}
\begin{minipage}[h]{0.49\hsize}
\centering
\includegraphics[scale=0.50]{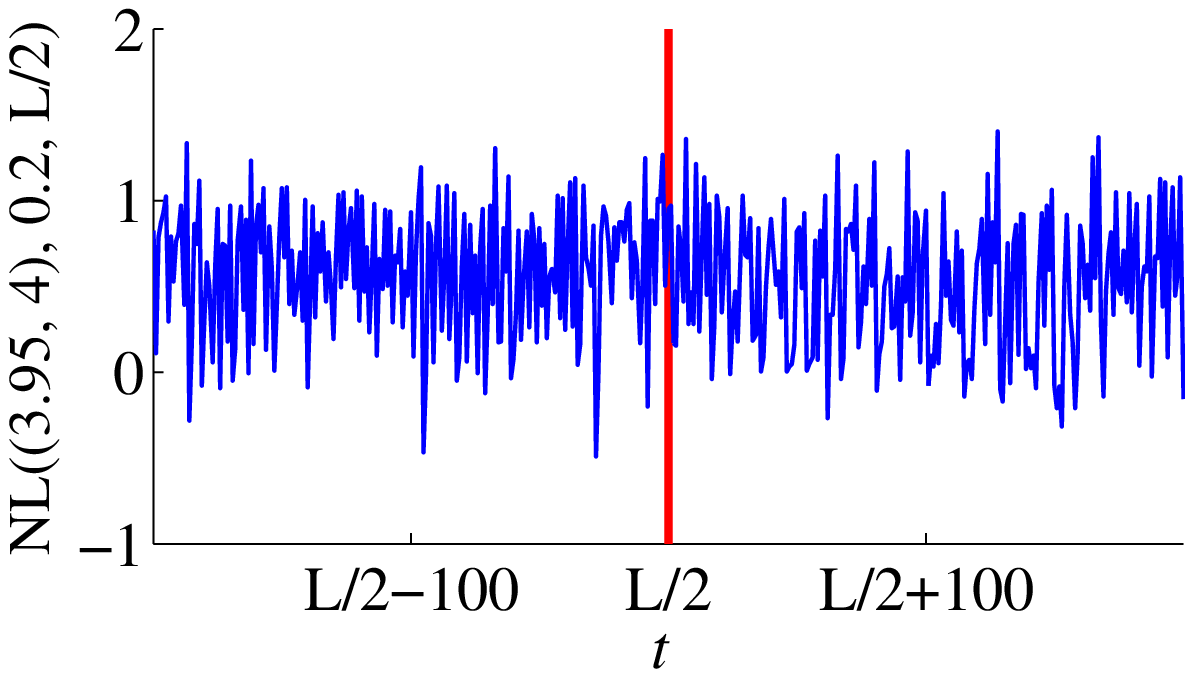}
	
\hspace{9mm}(b)

\includegraphics[scale=0.47]{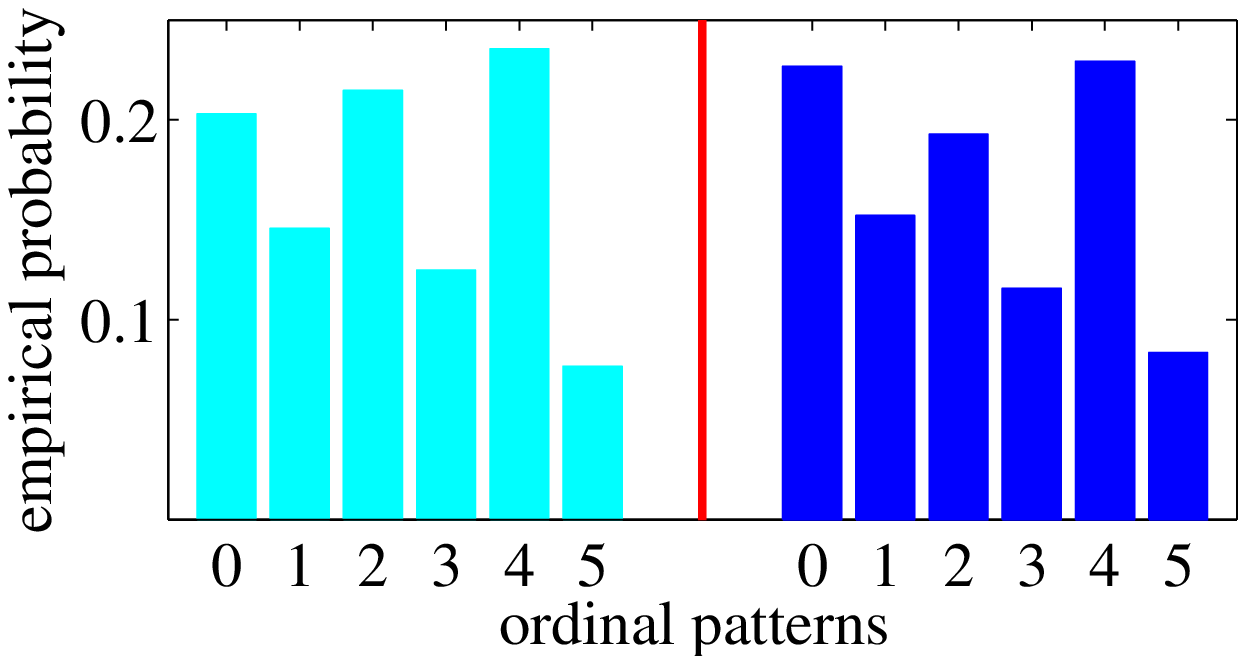}
	
\hspace{9mm}(d)
\end{minipage}
\caption{Upper row: parts of realizations of AR (a) and NL (b) process with change-points marked by vertical lines, $L=20000$.
Lower row: empirical probability distributions of ordinal patterns of order $d = 2$ before and after the change-point for the realizations of AR (c) and NL (d) process}
\label{OPdistrLogistAR}
\end{figure}

The NL and AR processes considered have rather different ordinal patterns distributions, being the reason for using them for empirical investigation of change-point detection methods in Section~\ref{SecExperiments}.

We now consider the classical problem of detecting a change-point $t^\ast$ on the basis of a realization $x$ of a stochastic process $X$ having at most one change-point, that is it holds either $N_\text{st} = 1$ or $N_\text{st} = 2$ (compare \cite{CarlsteinMullerSiegmund1994}).
To solve this problem one estimates a tentative change-point $\widehat{t}^\ast$ as the time-point that maximizes a test statistic $S(t; x)$.
Then the value of $S(\widehat{t}^\ast; x)$ is compared to a given threshold in order to decide whether $\widehat{t}^\ast$ is a change-point.

The idea of {\it ordinal change-point detection} is to find change-points in a stochastic process $X$ by detecting changes in the sequence $\pi^{d,L}$ of ordinal patterns for a realization of $X$.
Given at most one ordinal change-point $t^\ast$ in $X$, one estimates its position $\widehat{t}^\ast$ by using the fact that
\begin{itemize}
\item $\pi(d), 	\pi(d + 1),      \ldots, \pi(t^\ast)$     characterize the process before the change;
\item $\pi(t^\ast+1),  \pi(t^\ast+2),   \ldots, \pi(t^\ast+d-1)$ correspond to the transitional state;
\item $\pi(t^\ast+d),  \pi(t^\ast+d+1), \ldots, \pi(L)$ 	  characterize the process after the change.
\end{itemize}
Therefore, a position of a change-point can be estimated by an ordinal-patterns-based statistic $S(t; \pi^{d,L})$ that,
roughly speaking, measures dissimilarity between the distributions of ordinal patterns
for $\big(\pi(k)\big)_{k=d}^t$ and for $\big(\pi(k)\big)_{k=t+d}^L$.

Then an estimate of the change-point $t^\ast$ is given by
\begin{equation*}
\widehat{t}^\ast = \argmax_{t = d,d+1,\ldots,L} S(t; \pi^{d,L}).
\end{equation*}

A method for detecting one change-point can be extended to an arbitrary number of change-points using the binary segmentation \cite{Vostrikova1981}:
one applies a single change-point detection procedure to the realization $x$; if a change-point is detected then it splits $x$ into two segments in each of which one is looking for a change-point.
This procedure is repeated iteratively for the obtained segments until all of them either do not contain change-points or are too short.

The key problem in this paper is the selection of an appropriate ordinal-patterns-based test statistic $S(t; \pi^{d,L})$ for detecting changes.
We suggest an appropriate test statistic in Section~\ref{SecCEofOPstat}, but first we introduce in Subsection~\ref{subsec_CE} the conditional entropy of ordinal patterns being the cornerstone of this statistic.

\subsection{Conditional entropy of ordinal patterns}\label{subsec_CE}

Let us call a process $X = \big(X(t)\big)_{t=0}^L$ for $L\in {\mathbb N}\cup\{\infty\}$ \emph{ordinal-$d^+$-stationary} if for all $i,j\in S_d$ the {\it probability of pairs of ordinal patterns}
\begin{equation*}
p_{i,j}=\mathbb{P}\big(\Pi(t) = i,\Pi(t+1) = j\big)
\end{equation*}
does not depend on $t$ for $d\leq t \leq L-1$ (compare with Definition~\ref{OrdinalStationarity}).
Obviously, ordinal-$d+1$-stationarity implies ordinal-$d^+$-stationarity.

For an ordinal-$d^+$-stationary stochastic process,
consider the probability of an ordinal pattern $j \in S_d$ to occur after an ordinal pattern $i \in S_d$; similarly to \eqref{OPprob}, it is given by:
\begin{equation*}
p_{j|i}={\mathbb P}\big(\Pi(t+1) = j \mid \Pi(t) = i\big) = \frac{p_{i,j}}{p_i}
\end{equation*}
for $p_i\neq 0$. If $p_i=0$, let $p_{j|i}=0$.

\begin{definition}\label{CEdef}
The {\it conditional entropy of ordinal patterns} of order $d\in \mathbb{N}$ of an ordinal $d^+$-stationary stochastic process $X$ is defined by:
\begin {align}
  \CE(X, d) &= - \sum_{i \in S_d}\sum_{j \in S_d} p_{i}p_{j|i}\,\ln (p_{i}p_{j|i}) + \sum_{i \in S_d} p_{i}\,\ln p_{i} \nonumber\\
	    &= - \sum_{i \in S_d}\sum_{j \in S_d} p_{i}p_{j|i}\,\ln p_{j|i}.  \label{CEdefEq}
\end {align}
\end{definition}
For brevity, we refer to $\CE(X, d)$ as the ``conditional entropy'' when no confusion can arise.
The conditional entropy characterizes the mean diversity of successors $j \in S_d$ of a given ordinal pattern $i \in S_d$.
This quantity often provides a good practical estimation of the Kolmogorov-Sinai entropy for dynamical systems;
for a discussion of this and other theoretical properties of conditional entropy we refer to \cite{UnakafovKeller2014}.

One can estimate the conditional entropy from a time series by using the empirical conditional entropy of ordinal patterns \cite{KellerUnakafovUnakafova2014}.
Consider a sequence $\pi^{d,L}$ of ordinal patterns of order $d \in \mathbb{N}$ with length $L \in \mathbb{N}$.
Similarly to \eqref{OPfrequency}, the frequency of occurrence of an ordinal patterns pair $i,j \in S_d$ is given by
\begin{equation}\label{OPpairsFreq}
   n_{i,j}(t) = \#\{l \in \{d, d+1, \ldots, t-1\} \mid \pi(l) = i, \pi(l+1) = j \}
\end{equation}
for $t \in \{d+1, d+2,\ldots,L\}$.
The {\it empirical conditional entropy of ordinal patterns} for $\pi^{d,L}$ is defined by
 \begin{align}
   \eCE\big(\pi^{d,L}\big) =&-\frac{1}{L-d}\sum_{i \in S_d}\sum_{j \in S_d} n_{i,j}(L)\ln       n_{i,j}(L) +\frac{1}{L-d}\sum_{i \in S_d} n_{i}(L)\ln n_{i}(L)\nonumber\\
			  = &-\frac{1}{L-d}\sum_{i \in S_d}\sum_{j \in S_d} n_{i,j}(L)\ln \frac{n_{i,j}(L)}{n_{i}(L)}\label{empiricalCEeq}.
 \end{align}

As a direct consequence of Lemma~\ref{OPseqIsStat} the empirical conditional entropy approaches the conditional entropy under certain assumptions,
namely it holds the following.
\begin{corollary}\label{CEempiricalConv}
For the sequence $\pi^{d,\infty}$ of ordinal patterns of order $d \in \mathbb{N}$ of a realization of an ergodic stochastic process $X = \big(X(t) \big)_{t\in \mathbb{N}_0}$ with associated stationary increment process $(X(t) - X(t-1))_{t\in \mathbb{N}}$,
it holds almost surely that
\begin{equation}\label{CEempiricalSSP}
\lim_{L \to \infty}\eCE\Big(\big(\pi(k)\big)_{k=d}^L\Big) = \CE(X, d).
\end{equation}
\end{corollary}

\section{Statistic for change-point detection based of on the conditional entropy of ordinal patterns}\label{SecCEofOPstat}

We suggest to use the following statistic for detecting ordinal change-points of a process on the basis
of a sequence $\pi^{d,L}$ of ordinal patterns for $d, L \in \mathbb{N}$ of a realization of the process:
\begin{align}
\CEofOP\big(t; \pi^{d,L}\big) =&\,\,(L - 2d)\,\eCE\Big(\big(\pi(k)\big)_{k=d}^L\Big) \nonumber \\
&-   (t - d)\,\eCE\Big(\big(\pi(k)\big)_{k=d}^t\Big) \label{CEofOPstat} \\
&- \big(L -(t + d) \big)\,\eCE\Big(\big(\pi(k)\big)_{k=t+d}^L\Big).         \nonumber
\end{align}
for all $t\in \mathbb{N}$ with $d < t < L-d$.
The intuition behind this statistic comes from the concavity of conditional entropy (not only for ordinal patterns but in general, see \cite[Subsection~2.1.3]{HanKobayashi2002}).
It holds
\begin{align}
\eCE\Big(\big(\pi(k)\big)_{k=d}^L\Big) &\geq \frac{t - d}  {L - 2d}\, \eCE\Big(\big(\pi(k)\big)_{k=d}^t\Big)	\label{CEofOPJustif} \\
&+    \frac{L -(t + d)}{L - 2d}\, \eCE\Big(\big(\pi(k)\big)_{k=t+d}^L\Big). 		\nonumber
\end{align}
Therefore, if the probabilities of ordinal patterns change at some point $t^\ast$, but do not change before and after $t^\ast$, then $\CEofOP\big(t; \pi^{d,L}\big)$ tends to attain its maximum at $t = t^\ast$. If the probabilities do not change at all, then for $L$ being sufficiently large, \eqref{CEofOPJustif} tends to hold with equality
(see Corollary~\ref{CEofOP_nochange} in Section~\ref{sec_stationary}).

Let $n_{i}(t)$ and $n_{i,j}(t)$ be the frequencies of occurrence of an ordinal pattern $i \in S_d$ and of an ordinal patterns pair $i,j \in S_d$
(given by \eqref{OPfrequency} and \eqref{OPpairsFreq}, respectively).
Set $m_{i}(t) = n_{i}(L)- n_{i}(t + d)$ and $m_{i,j}(t)= n_{i,j}(L) - n_{i,j}(t + d)$,
then using \eqref{empiricalCEeq} we rewrite \eqref{CEofOPstat} in a straightforward form:
\begin{align}
\CEofOP\big(t; \pi^{d,L}\big) =&- \frac{L - 2d}{L - d} \sum_{i \in S_d}\sum_{j \in S_d} n_{i,j}(L) \ln \frac{n_{i,j}(L)}{n_{i}(L)} \label{CEofOPstraight}\\
&+ \sum_{i \in S_d}\sum_{j \in S_d} n_{i,j}(t) \ln \frac{n_{i,j}(t)}{n_{i}(t)}		
+ \sum_{i \in S_d}\sum_{j \in S_d} m_{i,j}(t) \ln \frac{m_{i,j}(t)}{m_{i}(t)}.	\nonumber
\end{align}
This statistic was first introduced and applied to the segmentation of sleep EEG time series in \cite{KellerUnakafovUnakafova2014}.

In the rest of this section, we present a motivating example (Subsection~\ref{subsec_motivating}),
show the relation of the $\CEofOP$ statistic to the likelihood ratio statistic for a piecewise stationary Markov chain (Subsection~\ref{subsec_Markov})
and formulate an algorithm for detecting multiple change-points by means of the CEofOP statistic in Subsection~\ref{subsecCPdetectionMethod}.

\subsection{Motivating example}\label{subsec_motivating}

We demonstrate a ``non-linear'' nature of the $\CEofOP$ statistic by means of Example~\ref{CEofOPsurrogateEx} concerning transition  from a time series to its surrogate. 
Although being in a sense tailor-made, this example shows that $\CEofOP$ discerns changes that cannot be detected by conventional ``linear'' methods.
 
\begin{remark}
The question whether a time series is linear or non-linear often arises in data analysis.
For instance, linearity should be verified before using such powerful methods as Fourier analysis. 
For this one usually employs a procedure known as surrogate data testing \cite{Theiler1992,SchreiberSchmitz1996,SchreiberSchmitz2000}. 
It utilises the fact that a linear time series is statistically indistinguishable from any time series sharing some of its properties (for instance, second moments and amplitude spectrum).
Therefore one can generate surrogates having the certain properties of the original time series without preserving other properties, irrelevant for a linear system. 
If such surrogates are significantly different from the original series then non-linearity is assumed.
\end{remark}

\begin{example}\label{CEofOPsurrogateEx}
Consider a time series consisting of a realisation of a noisy logistic process $\NL\big(r, \sigma\big)$ of length $L/2$ without changes, glued with its surrogate of the same length (to generate surrogates we use the iterative AAFT algorithm suggested by \cite{SchreiberSchmitz1996} and implemented by \cite{Gautama2005}).
This compound time series has a change-point at $t^\ast = L/2$, which conventional methods may fail to detect since the surrogate has the same autocorrelation function as the original process (for instance this is the case for the Brodsky-Darkhovsky method considered further in Section~\ref{SecExperiments}). 
However, the ordinal pattern distributions for the original time series and its surrogate generally are significantly different.
Therefore the $\CEofOP$ statistic detects the change-point, which is illustrated by Figure~\ref{surrogateTest_Figure}.

\begin{figure}[!th]
\centering
\includegraphics[scale=0.45]{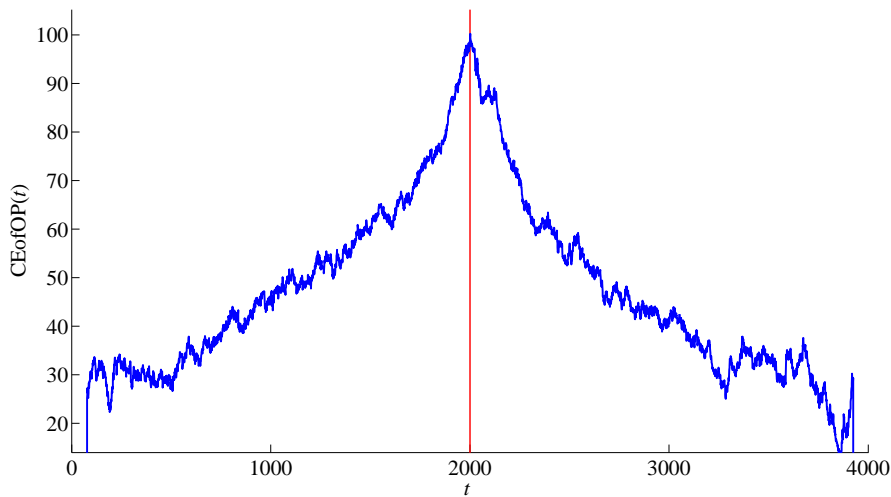}
\caption{Statistic $\CEofOP(\theta L)$ for a time series, obtained by ``gluing'' a realization of a noisy logistic stochastic process $\NL\big(4, 0.2\big)$ with its surrogate at $t^\ast = 2000$}
\label{surrogateTest_Figure}
\end{figure}
\end{example}

\begin{remark}
Although the idea that ordinal structure is a relevant indicator of time series linearity/non-linearity is not new
\cite{BandtPompe2002, Amigo2010}, to our knowledge, it was not rigorously proved that the distribution of ordinal patterns is altered by surrogates.
This is clearly beyond the scope of this paper and will be discussed elsewhere as a separate study; here it is sufficient for us to provide an empirical evidence for this. 
\end{remark}

\subsection{CEofOP statistic for a sequence of ordinal patterns forming a Markov chain}\label{subsec_Markov}

In this subsection we show that there is a connection between the $\CEofOP$ statistic and the classical likelihood ratio statistic.
Though taking place only in a particular case, this connection reveals the nature of the $\CEofOP$ statistic.

First we set up necessary notations.
Consider a sequence $\pi^{d,L}$ of ordinal patterns for that transition probabilities of ordinal patterns may change at some $t \in \{d,d+1\ldots,L\}$.
The basic statistic for testing whether there is a change in the transition probabilities is the likelihood ratio statistic \cite[Subsection~2.2.3]{BassevilleNikiforov1993}:
\begin{equation}\label{likelihoodRatioGeneraleq}
\mathrm{LR}\big(t;\pi^{d,L}\big)= -2\ln\Likelihood\big(H_0 \mid \pi^{d,L}\big) + 2\ln\Likelihood\big(H_A\mid \pi^{d,L}\big),
\end{equation}
where $\Likelihood\big(H \mid \pi^{d,L}\big)$ is the likelihood of the hypothesis $H$ given a sequence $\pi^{d,L}$ of ordinal patterns, and the hypotheses are given by
\begin{align*}
H_0&: \big(p_{j|i}(t)\big)_{i,j \in S_d}   =  \big(q_{j|i}(t)\big)_{i,j \in S_d},\\
H_A&: \big(p_{j|i}(t)\big)_{i,j \in S_d} \neq \big(q_{j|i}(t)\big)_{i,j \in S_d},
\end{align*}
where $p_{j|i}(t)$, $q_{j|i}(t)$ are transition probabilities of ordinal patterns before and after $t$, respectively.

\begin{proposition}\label{MarkovLikelihood}
If an abstract sequence $\Pi^d$ of ordinal patterns of order $d\in \mathbb{N}$ forms a Markov chain with at most one change-point,
then for a sequence $\pi^{d,L} = \big(\pi(k)\big)_{k=d}^L$ of ordinal patterns being a realization of $\Pi^d$  of length $L\in \mathbb{N}$, it holds
\begin{equation*}
\mathrm{LR}\big(t; \pi^{d,L}\big) = 2\,\CEofOP\big(t; \pi^{d,L}\big) + 2d \cdot \eCE\big( \pi^{d,L} \big).
\end{equation*}
\end{proposition}
\begin{proof}
First we estimate the probabilities and the transition probabilities before ($p$) and after ($q$) the change \cite[Section~2]{AndersonGoodman1957}:
\begin{align*}
\widehat{p}_{i}(t)   &= \frac{n_{i}  (t)}{t - d  },\,\,\,\,\,\,\,\,\,\,\,\,\,\,\,\,\,\,\,\,\,\,\,\,   \widehat{p}_{j|i}(t) = \frac{n_{i,j}(t)}{n_{i}(t)},\\
\widehat{q}_{i}(t)   &= \frac{m_{i}  (t)}{L - (t + d)},\,\,\,\,\,\,\,\,    		            \widehat{q}_{j|i}(t) = \frac{m_{i,j}(t)}{m_{i}(t)}.
\end{align*}
Then, as one can see from \cite[Section~3.2]{AndersonGoodman1957}, we have
\begin{align*}
\Likelihood\big(H_0 \mid \pi^{d,L}\big)  &= \widehat{p}_{\pi(d)}(L) \prod_{l = d}    ^ {L-1} \widehat{p}_{\pi(l+1)|\pi(l)}(L)\\
&= \widehat{p}_{\pi(d)}(L) \prod_{i \in S_d}\prod_{j \in S_d} \big(\widehat{p}_{j|i}(L)\big)^{n_{i,j}(L)},\\
\Likelihood\big(H_A\mid \pi^{d,L}\big)  &= \widehat{p}_{\pi(d)}(t) \prod_{l = d}    ^ {t} \widehat{p}_{\pi(l+1)|\pi(l)}(t)
\prod_{l = t + d}^ {L-1} \widehat{q}_{\pi(l+1)|\pi(l)}(t)\\
&= \widehat{p}_{\pi(d)}(t) \prod_{i \in S_d}\prod_{j \in S_d} \big(\widehat{p}_{j|i}(t)\big)^{n_{i,j}(t)}
\prod_{i \in S_d}\prod_{j \in S_d} \big(\widehat{q}_{j|i}(t)\big)^{m_{i,j}(t)}.
\end{align*}

Assume that the first ordinal pattern $\pi(d)$ is fixed in order to simplify the computations.
Then $\widehat{p}_{\pi(d)}(L) = \widehat{p}_{\pi(d)}(t)$ and it holds:
\begin{align*}
\mathrm{LR}\big(t; \pi^{d,L}\big) = -&2\sum_{i \in S_d}\sum_{j \in S_d} n_{i,j}(L) \big(\ln n_{i,j}(L) - \ln n_{i}(L)\big)\\
+&2\sum_{i \in S_d}\sum_{j \in S_d} n_{i,j}(t) \big(\ln n_{i,j}(t) - \ln n_{i}(t)\big)\\
+&2\sum_{i \in S_d}\sum_{j \in S_d} m_{i,j}(t) \big(\ln m_{i,j}(t) - \ln m_{i}(t)\big).
\end{align*}		
Since $\sum\limits_{j \in S_d} n_{i,j}(t) = n_{i}(t)$, one finally obtains:
\begin{align*}
\mathrm{LR}\big(t; \pi^{d,L}\big) =&  2 (L-d) 	      	\, \eCE\Big(\big(\pi(k)\big)_{k=d}^L\Big)\\
&- 2 (t-d)   	        \, \eCE\Big(\big(\pi(k)\big)_{k=d}^t\Big)\\ 	
&- 2\big(L-(t + d) \big)\, \eCE\Big(\big(\pi(k)\big)_{k=t+d}^L\Big)\\
=&  2\,\CEofOP\big(t; \pi^{d,L}\big) + 2d \cdot \eCE\Big(\pi^{d,L} \Big).\qedhere
\end{align*}
\end{proof}

\subsection{Change-point detection via the CEofOP statistic}\label{subsecCPdetectionMethod}

Consider a sequence $\pi^{d,L}$ of ordinal patterns of order $d \in \mathbb{N}$ with length $L \in \mathbb{N}$, corresponding to a realization of some piecewise stationary stochastic process.
To detect a single change-point via the $\CEofOP$ statistic we first estimate its possible position by
\begin{equation*}
\widehat{t}^\ast = \argmax_{t = T_{\text{min}} + d, \ldots, L- T_{\text{min}}} S(t; \pi^{d,L}),
\end{equation*}
where $T_{\text{min}}$ is a minimal length of a sequence of ordinal patterns that is sufficient for a reliable estimation of empirical conditional entropy.
\begin{remark}
From the representation \eqref{CEofOPstat} it follows that for a reasonable computation of the $\CEofOP$ statistic, a reliable estimation of $\eCE$ before and after the assumed change-point is required.
To satisfy this requirement the stationary parts of a process are assumed to be sufficiently long.
We take $T_{\text{min}} = (d+1)!(d+1)$, which is equal to the number of all possible pairs of ordinal patterns of order $d$ (see \cite{KellerUnakafovUnakafova2014} for details).
Consequently, the length $L$ of a time series should satisfy
\begin{equation}\label{LengthLowBoundary}
L > 2T_{\text{min}} = 2(d+1)!(d+1).
\end{equation}
Note that this does not impose serious limitations on the suggested method, since
condition \eqref{LengthLowBoundary} is not too restrictive for $d \leq 3$.
However it implies using of either $d=2$ or $d=3$, since $d=1$ does not provide effective change-point detection (see Examples~\ref{CEofOPARprocess} and~\ref{ExampleARwithChanges}), while $d>3$ in most applications demands too large sample sizes.
\end{remark}
Then in order to check whether $\widehat{t}^\ast$ is an actual change-point we test between the hypotheses:
\begin{description}
\item[$H_0$:] parts $\pi(d), \pi(d+1), \ldots, \pi(\widehat{t}^\ast)$ and $\pi(\widehat{t}^\ast+d), \ldots, \pi(L)$ of the sequence $\pi^{d,L}$ come from the same distribution;
\item[$H_A$:] parts $\pi(d), \pi(d+1), \ldots, \pi(\widehat{t}^\ast)$ and $\pi(\widehat{t}^\ast+d), \ldots, \pi(L)$ of the sequence $\pi^{d,L}$ come from different distributions.
\end{description}
This test is performed by comparing $\CEofOP\mleft(\widehat{t}^\ast; \pi^{d,L}\mright)$ to a threshold $h$,
such that if the value of $\CEofOP$ is above the threshold then one rejects $H_0$ in favour of $H_A$, otherwise $H_0$ is accepted.
The choice of the threshold is ambiguous:
the higher $h$, the higher the possibility of false acceptance of the hypothesis $H_0$ is;
on the contrary, the lower $h$, the higher the possibility of false rejection of $H_0$ in favour of
$H_A$ ({\it false alarm}) is.

As it is usually done, we consider the threshold $h$ as a function of the desired probability $\alpha$ of false alarm;
for computing the threshold $h(\alpha)$ we use block bootstrapping from the sequence $\pi^{d,L}$ of ordinal patterns
(bootstrapping is often used in change-point detection for computing a threshold, see \cite{DavisonHinkley1997, Lahiri2003} for a theoretical discussion
and \cite{Polansky2007, KimMarzbanPercivalStuetzle2009} for applications of bootstrapping with detailed and clear explanations).
Algorithm~\ref{Problems2flowchart} describes the detection of at most one change-point via the $\CEofOP$ statistic.

\begin{algorithm}[!th]
\caption{Detecting at most one change-point}\label{Problems2flowchart}
\begin{algorithmic}[1]
\Require{sequence $\pi = \big(\pi(k)\big)_{k=t_\text{start}}^{t_\text{end}}$ of ordinal patterns of order $d$,
nominal probability $\alpha$ of false alarm}
\Ensure{estimate of a change-point $\widehat{t}^\ast$ if change-point is detected, otherwise return $0$.}
\Function{DetectSingleCP}{$\pi$, $\alpha$}
\Let{$T_{\text{min}}$}{$(d+1)!(d+1)$};
\If{$t_\text{end} - t_\text{start} < 2T_{\text{min}}$}
\State \Return{0};  \Comment{sequence is too short, no change-point can be detected}
\EndIf
\Let{$\widehat{t}^\ast$}{$\argmax\limits_{t = t_\text{start} + T_{\text{min}}, \ldots, t_\text{end} - T_{\text{min}}} \CEofOP(t; \pi)$};

\Let{$N_{\text{boot}}$}{$\lfloor \frac{5}{\alpha} \rfloor$}; \Comment{number of bootstrap samples for computing threshold}	
\For{$l = 1,2, \ldots, N_{\text{boot}}$}	 \Comment{computing threshold by bootstrapping}					
	
\Let{$\xi$}{randomly shuffled blocks of length $(d+1)$ from $\pi$};
\Let{$c_j$}          {$\argmax\limits_{t = t_\text{start} + T_{\text{min}}, \ldots, t_\text{end} - T_{\text{min}}} \CEofOP(t; \xi)$};
\EndFor
\Let{$c_j$}{Sort($c_j$);}			\Comment{sort the maximal values of $\CEofOP$ for bootstrap samples in decreasing order}		
\Let{$h$}{$c_{\lfloor\alpha N_{\text{boot}}\rfloor}$}

\If{$S(\widehat{t}^\ast; \pi) < h$}
\State \Return{0};
\Else
\State \Return{$\widehat{t}^\ast$};
\EndIf
\EndFunction
\end{algorithmic}
\end{algorithm}

To detect multiple change-points we use an algorithm that consists of two steps:
\begin{description}
\item [Step 1:] preliminary estimation of boundaries of the stationary segments with a threshold $h(2\alpha)$ computed for doubled nominal probability of false alarm
(that is with a higher risk of detecting false change-points).
\item [Step 2:] verification of the boundaries and exclusion of false change-points: a change-point is searched for a merging of every two adjacent intervals.
\end{description}
Details of these two steps are displayed in Algorithm~\ref{Problems3flowchart}.
While Step~1 is the usual binary segmentation procedure as suggested in \cite{Vostrikova1981}, Step~2 improves the obtained solution following the idea suggested in \cite{BrodskyDarkhovskyKaplanShishkin1999}.

\begin{algorithm}[!th]
\caption{Detecting multiple change-points}\label{Problems3flowchart}
\begin{algorithmic}[1]
\Require{sequence $\pi = \big(\pi(k)\big)_{k=d}^L$ of ordinal patterns of order $d$, nominal probability $\alpha$ of false alarm.}
\Ensure{estimates of the number $\widehat{N}_\text{st}$ of stationary segments and of their boundaries $\big(\widehat{t}^\ast_k\big)_{k=0}^{\widehat{N}_\text{st}}$.}
\Statex
\Function{DetectAllCP}{$\pi$, $\alpha$}
\State $\widehat{N}_\text{st} \gets 1$; $\widehat{t}^\ast_0 \gets 0$; $\widehat{t}^\ast_1 \gets L$; $k \gets 0$  \Comment{Step 1} 	
\Repeat 		
\Let{$\widehat{t}^\ast$}{\LCall{DetectSingleCP}{$\big(\pi(k)\big)_{k=\widehat{t}^\ast_k + d}^{\widehat{t}^\ast_{k+1}}$, $2\alpha$}};
\If{$\widehat{t}^\ast > 0$}
\State   {\bf Insert }$\widehat{t}^\ast$ to the list of change-points after $\widehat{t}^\ast_i$;
\Let{$\widehat{N}_\text{st}$}{$\widehat{N}_\text{st}+1$};
\Else
\Let{$k$}{$k+1$};
\EndIf
\Until{$k < \widehat{N}_\text{st}$;}
\Let{$k$}{$0$};   	\Comment{Step 2}
\Repeat
\Let{$\widehat{t}^\ast$}{\LCall{DetectSingleCP}{$\big(\pi(k)\big)_{k=\widehat{t}^\ast_k + d}^{\widehat{t}^\ast_{k+2}}$,
$\alpha$}};
\If{$\widehat{t}^\ast > 0$}
\Let{$\widehat{t}^\ast_{k+1}$}{$\widehat{t}^\ast$};
\Let{$k$}{$k+1$};
\Else
\State   {\bf Delete }$\widehat{t}^\ast_{k+1}$ from the change-points list;
\Let{$\widehat{N}_\text{st}$}{$\widehat{N}_\text{st}-1$};
\EndIf
\Until{$k < \widehat{N}_\text{st} - 1$;}
\State \Return{$\widehat{N}_\text{st}, \big(\widehat{t}^\ast_k\big)_{k=0}^{\widehat{N}_\text{st}}$};
\EndFunction
\end{algorithmic}
\end{algorithm}

\section{Experiments}\label{SecExperiments}

In this section we empirically investigate performance of the method for change-point detection via the $\CEofOP$ statistic.
We apply it to the noisy logistic processes and to autoregressive processes (see Subsection~\ref{subsec_CPinOPdetect}) and
compare performances of change-point detection by the suggested method and by the following methods.
\begin{itemize}
\item The ordinal-patterns-based method for detecting change-points via the {\it CMMD statistic} \cite{SinnGhodsiKeller2012, SinnKellerChen2013}:
A time series is split into {\it windows} of equal lengths $W \in \mathbb{N}$, empirical probabilities of ordinal patterns are estimated in every window.
If there is a ordinal change-point in the time series,
then the empirical probabilities of ordinal patterns should be approximately constant before the change-point and after the change-point,
but they change at the window with the change-point.
To detect this change authors have introduced the CMMD statistic\footnote{Note that the definition of the CMMD statistic in \cite{SinnGhodsiKeller2012} contains a mistake,
which is corrected in \cite{SinnKellerChen2013}.
The results of numerical experiments reported in \cite{SinnGhodsiKeller2012} also do not comply
with the actual definition of the CMMD statistic,
see \cite[Subsections~4.2.1.1,~4.5.1.1]{Unakafov2015} for details.}.
In the original papers \cite{SinnGhodsiKeller2012, SinnKellerChen2013} authors do not estimate change-points, but only the corresponding window numbers;
for the algorithm of change-point estimation by means of the CMMD statistic we refer to \cite[Subsection~4.5.1]{Unakafov2015}.									
\item Two versions of the {\it classical Brodsky-Darkhovsky method} \cite{BrodskyDarkhovskyKaplanShishkin1999}:
The Brodsky-Darkhovsky method can be used for detecting changes in various characteristics of a time series $x = \big(x(t)\big)_{t=1}^L$,
but the characteristic of interest should be selected in advance.
In this paper we consider detecting changes in mean which is just the basic characteristic,
and in correlation function $\operatorname{corr}(x(t), x(t+1))$
which reflects relations between the future and the past of a time series and seems to be a natural choice for detecting ordinal change-points.
Changes in mean are detected by the generalized version of the Kolmogorov-Smirnov statistic \cite{BrodskyDarkhovskyKaplanShishkin1999}:
\begin{equation*}
\mathrm{BD^{exp}}(t; x, \delta) = \mleft(\frac{t(L-t)}{L^2}\mright)^\delta \: \mleft|\frac{1}{t}\sum_{l=1}^t x(l) - \frac{1}{L-t}\sum_{l=t+1}^L x(l)\mright|,
\end{equation*}
where the parameter $\delta \in [0, 1]$ regulates properties of the statistic, $\delta = 0$ is basically used (see \cite{BrodskyDarkhovskyKaplanShishkin1999} for details).
Changes in the correlation function are detected by the following statistic:
\begin{equation*}
\mathrm{BD^{corr}}(t; x, \delta) = \mathrm{BD^{exp}}\big(t; \big(y(t)\big)_{t=1}^{L-1}, \delta\big) \text{ with } y(t) = x(t)x(t+1).
\end{equation*}
\end{itemize}
\begin{remark}
Note that we consider the statistic $\mathrm{BD^{exp}}$, which is intended to detect changes in mean, though ordinal-patterns-based statistics do not detect these changes. This is motivated by the fact that changes in the noisy logistic processes are on the one hand changes in mean, and in the other hand -- ordinal changes in the sense of Definition~\ref{StructureChanges_def}. Therefore, they can be detected both by $\mathrm{BD^{exp}}$ and by ordinal-patterns-based statistics. In general, by the nature of ordinal time series analysis, changes in mean and in the ordinal structure are in some sense complementary. 
\end{remark}

We carry out experiments for order $d = 3$ of ordinal patterns
(lower orders may provide worse results because of reduced sensitivity, while higher orders are applicable only to rather long time series due to condition \eqref{LengthLowBoundary}).
For the CMMD statistic; we take the window size $W = 256$.
(There are no special reasons for this choice except the fact that $W = 256$ is sufficient for estimating probabilities of ordinal patterns of order $d = 3$
inside the windows, since $256 > 120  = 5(d+1)!$ \cite[Section 9.3]{Amigo2010}.
Results of the experiments remain almost the same for $200\leq W\leq1000$.)

In Subsection~\ref{sec4_ExpEstimation} we study how well the statistics for change-point detection estimate the position of a single change-point.
Since we expect that the performance of the statistics for change-point detection may strongly depend on the length of realization,
we check this in Subsection~\ref{sec4_ExpEstimationSample}.
Finally, we investigate the performance of various statistics for detecting multiple change-points in Subsection~\ref{sec4_ExpMultipleCP}.

\subsection{Estimation of the position of a single change-point }\label{sec4_ExpEstimation}

Consider $N = 10000$ realizations $x^j = \big(x^j(t)\big)_{t=0}^L$ with $j = 1, \ldots, N$ for each of the processes listed in Table~\ref{CPtimeSeries1_tbl}.
A single change occurs at a random time $t^\ast$ uniformly distributed in $\mleft\{\frac{L}{4} - W, \frac{L}{4} - W + 1, \ldots, \frac{L}{4} + W\mright\}$.
For all processes, length $L = 80\,W$ of sequences of ordinal patterns is taken.

\begin{table}[!th]
\centering
\begin{tabular}{ l | l }
\toprule
Short name & Complete designation\\
\midrule
NL, $3.95\rightarrow3.98$, $\sigma=0.2$& $\mathrm{NL}\big((3.95, 3.98), (0.2, 0.2), t^\ast\big)$\\
NL, $3.95\rightarrow3.80$, $\sigma=0.3$& $\mathrm{NL}\big((3.95, 3.80), (0.3, 0.3), t^\ast\big)$\\
NL, $3.95\rightarrow4.00$, $\sigma=0.2$& $\mathrm{NL}\big((3.95, 4.00), (0.2, 0.2), t^\ast\big)$\\
\midrule
AR, $0.1\rightarrow 0.3$ 				& $\mathrm{AR}\big((0.1, 0.3), t^\ast\big)$\\
AR, $0.1\rightarrow 0.4$ 				& $\mathrm{AR}\big((0.1, 0.4), t^\ast\big)$\\
AR, $0.1\rightarrow 0.5$				& $\mathrm{AR}\big((0.1, 0.5), t^\ast\big)$\\
\bottomrule
\end{tabular}\vspace{2mm}
\caption{Processes used for investigation of the change-point detection}
\label{CPtimeSeries1_tbl}
\end{table}

To measure the overall accuracy of change-point detection via some statistic $S$ as applied to the process $X$ we use three quantities.
Let us first determine the error of the change-point estimation provided by the statistic $S$ for the $j$-th realization of a process $X$:
\begin{equation*}
\mathrm{err}^j(S, X) = \Big( \widehat{t}^\ast(S; x^j) - t^\ast \Big),
\end{equation*}
where $t^\ast$ is the actual position of the change-point and $\widehat{t}^\ast(S; x^j)$ is its estimate obtained by using $S$.
Then the {\it fraction of satisfactorily estimated change-points} $\mathrm{sE}$ (averaged over $N$ realizations) is defined by:
\begin{equation*}
\mathrm{sE}(S, X) =  \frac{\#\big\{j \in \{1,2,\ldots, N\}:\, |\mathrm{err}^j(S, X)| \leq \mathrm{MaxErr} \big\}}{N},
\end{equation*}
where $\mathrm{MaxErr}$ is the maximal satisfactory error, we take $\mathrm{MaxErr} = W = 256$.
The {\it bias} and the {\it root mean squared error} are respectively given by
\begin{align*}
\mathrm{B}(S, X) &= \frac{1}{N}\sum\limits_{j=1}^{N} \mathrm{err}^j(S, X),\\
\RMSE(S, X) 	 &= \sqrt{ \frac{1}{N} \sum\limits_{j=1}^{N}\big(\mathrm{err}^j(S, X)\big)^2 }.
\end{align*}
The larger $\mathrm{sE}$ is and the more near to zero the bias and the $\RMSE$ are, the more accurate the estimation of a change-point is.

Results of the experiments are presented in Tables~\ref{CPproblem1_log_tbl},~\ref{CPproblem1_ar_tbl} for NL and AR processes, respectively.
For every process the best values of performance measures are shown in {\bf bold}.

\begin{table}[!th]
\centering
\begin{tabular}{ l | c c c | c c c | c c c }
\toprule
&\multicolumn{3}{c|}{NL, $3.95\rightarrow3.98$} & \multicolumn{3}{c|}{NL, $3.95\rightarrow3.80$}& \multicolumn{3}{c}{NL, $3.95\rightarrow4.00$}\\
Statistic 		  &\multicolumn{3}{c|}{$\sigma=0.2$}              & \multicolumn{3}{c|}{$\sigma=0.3$}             & \multicolumn{3}{c}{$\sigma=0.2$}\\
\cmidrule(r){2-10}
&$\mathrm{sE}$ &$\mathrm{B}   $&$\RMSE$         & $\mathrm{sE}$&$\mathrm{B}   $&$\RMSE$        & $\mathrm{sE}$&$\mathrm{B}   $&$\RMSE$\\
\midrule
$\mathrm{CMMD}$       &      0.34    &      698      &        1653    &    0.50      &       -51     &      306      &      0.68    &   {\bf-13}    &     206\\
$\mathrm{CEofOP}$     &      0.61    &  \,{\bf53}    &     \,\,397    &    0.65      &\,\,\,\,{\bf1} &      256      &      0.88    &    \,\,20     &      99\\
$\mathrm{BD^{exp}}$   &  {\bf0.62}   &   \,\,78      &\,\,{\bf 351}   &{\bf0.78}      &   \,\,-6     &  {\bf145}     & {\bf0.89}    &    \,\,43     &\,{\bf 96}\\
$\mathrm{BD^{corr}}$  &      0.44    &   \,\,85      &     \,\,656    &    0.71      &    \,\,13     &      202      &      0.77    &    \,\,43     &     189\\
\bottomrule
\end{tabular}\vspace{2mm}
\caption{Performance of different statistics for estimating change-point in NL processes}
\label{CPproblem1_log_tbl}
\end{table}

\begin{table}[!th]
\centering
\begin{tabular}{ l | c c c | c c c | c c c }
\toprule
Statistic	          &\multicolumn{3}{c|}{AR, $0.1\rightarrow0.3$}  & \multicolumn{3}{c|}{AR, $0.1\rightarrow 0.4$}& \multicolumn{3}{c}{AR, $0.1\rightarrow 0.5$}\\
\cmidrule(r){2-10}
&$\mathrm{sE}$&$\mathrm{B}   $&$\RMSE$         & $\mathrm{sE}$&$\mathrm{B}   $&$\RMSE$        &  $\mathrm{sE}$&$\mathrm{B}   $&$\RMSE$\\
\midrule
$\mathrm{CMMD}$       &      0.32   &      616      &      1626      &     0.54     &        -14    &      368      &      0.68     &        -48    &    184\\
$\mathrm{CEofOP}$     &      0.39   &      126      &      1838      &     0.68     &\,\,\,\,{\bf0} &      234      &      0.86     &\,\,\,\,{\bf0} &    110\\
$\mathrm{BD^{exp}}$   &      0.00   &\small{>$10^3$}&\small{>$10^4$}&     0.00      &\small{>$10^4$}&\small{>$10^4$}&      0.00     &\small{>$10^4$}&\small{>$10^4$}\\
$\mathrm{BD^{corr}}$  &  {\bf0.79}  & \,{\bf31}     &\,\,{\bf151}    &{\bf 0.92}    &     \,\,21    &  \,{\bf73}    & {\bf 0.97}    &     \,\,21    &\,{\bf50}\\
\bottomrule
\end{tabular}\vspace{2mm}
\caption{Performance of different statistics for estimating change-point in AR processes}
\label{CPproblem1_ar_tbl}
\end{table}

Let us summarize:
For the considered processes the CEofOP statistic estimates change-point more accurately than the CMMD statistic.
For the NL processes the CEofOP statistic has almost the same performance as the Brodsky-Darkhovsky method;
for the AR processes performance of the classical method is better, though CEofOP has lower bias.
In contrast to the ordinal-patterns-based methods,
the Brodsky-Darkhovsky method is unreliable when there is lack of a priori information about the time series.
For instance, changes in NL processes only slightly influence the correlation function and $\mathrm{BD^{corr}}$ does not provide a good indication of changes
(cf. performance of $\mathrm{BD^{corr}}$ and $\CEofOP$ in Table~\ref{CPproblem1_log_tbl}).
Meanwhile, changes in the AR processes do not influence the expected value (see Example~\ref{ExampleARwithChanges}),
which does not allow to detect them using $\mathrm{BD^{exp}}$ (see Table~\ref{CPproblem1_ar_tbl}).
Therefore we do not consider the $\mathrm{BD^{exp}}$ statistic in further experiments.

\subsection{Estimating position of a single change-point for different lengths of time series}\label{sec4_ExpEstimationSample}

Here we study how the accuracy of change-point estimation for the three considered statistics depends on the length $L$ of a time series.
We take $N = 50000$ realizations of NL, $3.95\rightarrow3.98$, $\sigma=0.2$ and AR, $0.1\rightarrow 0.4$ for realization lengths $L = 24\,W, 28\,W, \ldots, 120\,W$.
Again, we consider a single change at a random time $t^\ast \in \mleft\{\frac{L}{4} - W, \frac{L}{4} - W + 1, \ldots, \frac{L}{4} + W\mright\}$.
Results of the experiment are presented in Figure~\ref{CPproblem1sample_fig}.
\begin{figure}[!th]
\begin{minipage}[h]{0.49\hsize}
\centering
\includegraphics[scale=0.43]{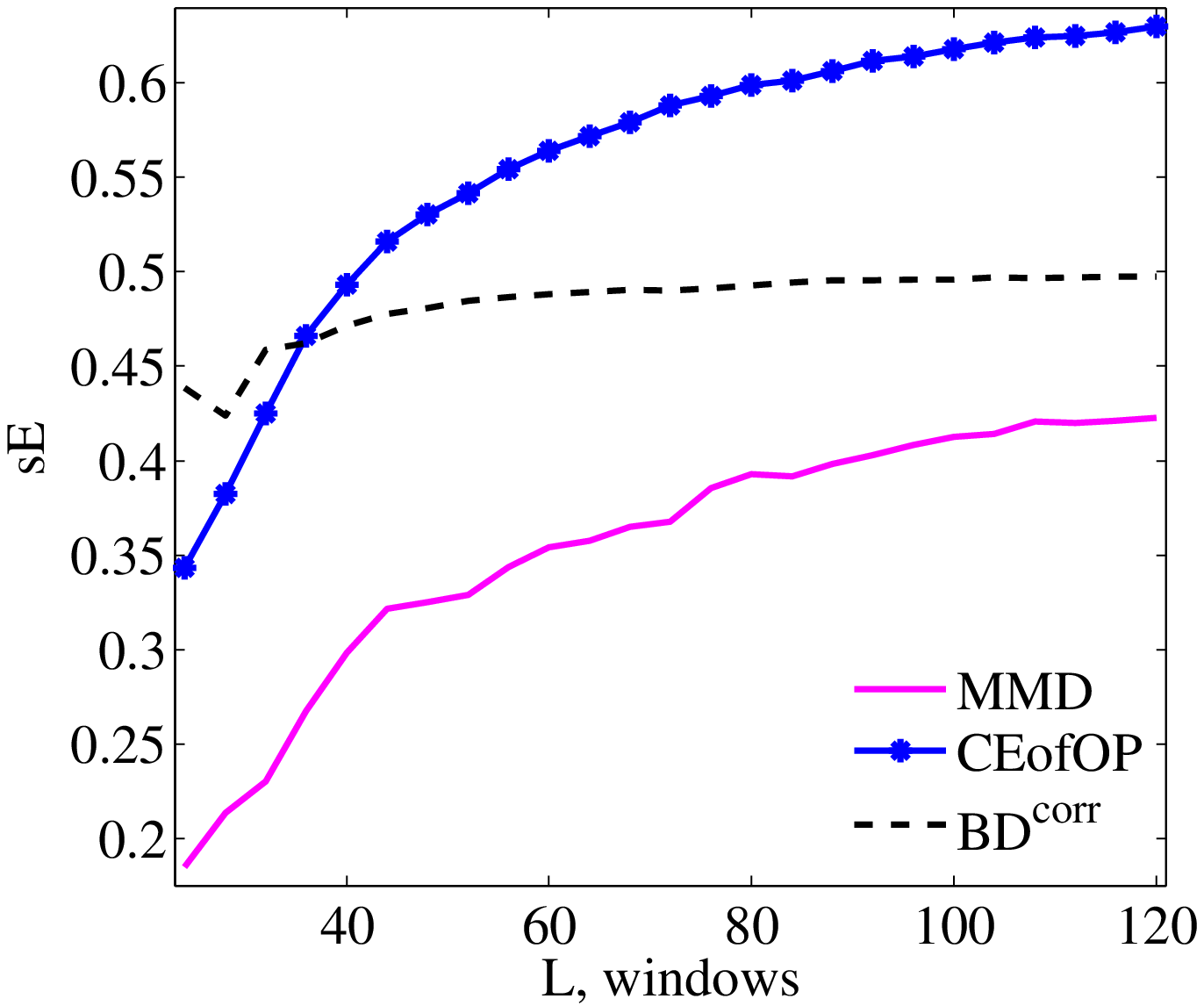}
	
\hspace{5mm}(a)

\includegraphics[scale=0.43]{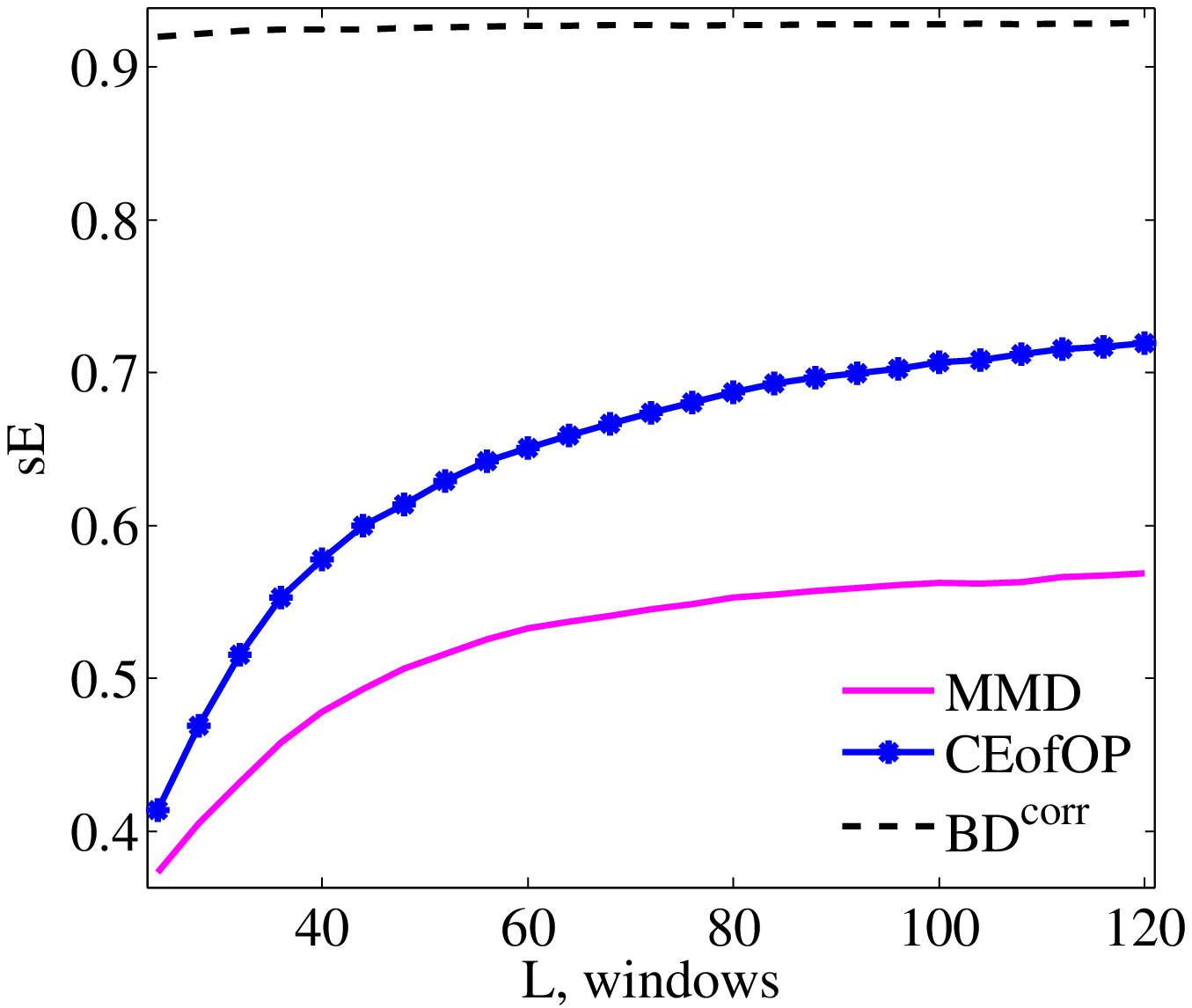}
	
\hspace{5mm}(c)
\end{minipage}
\begin{minipage}[h]{0.49\hsize}
\centering
\includegraphics[scale=0.43]{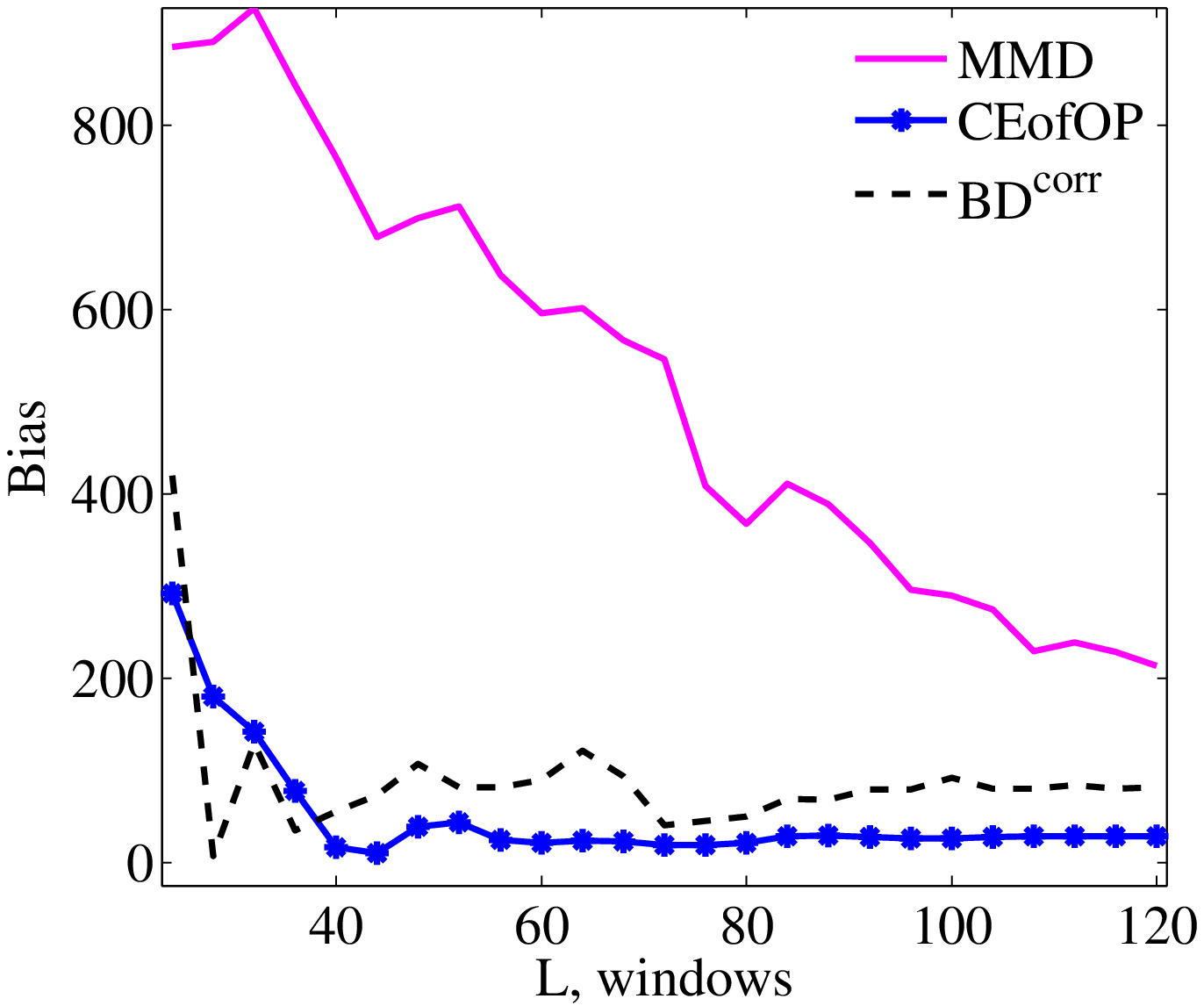}
	
\hspace{5mm}(b)

\includegraphics[scale=0.43]{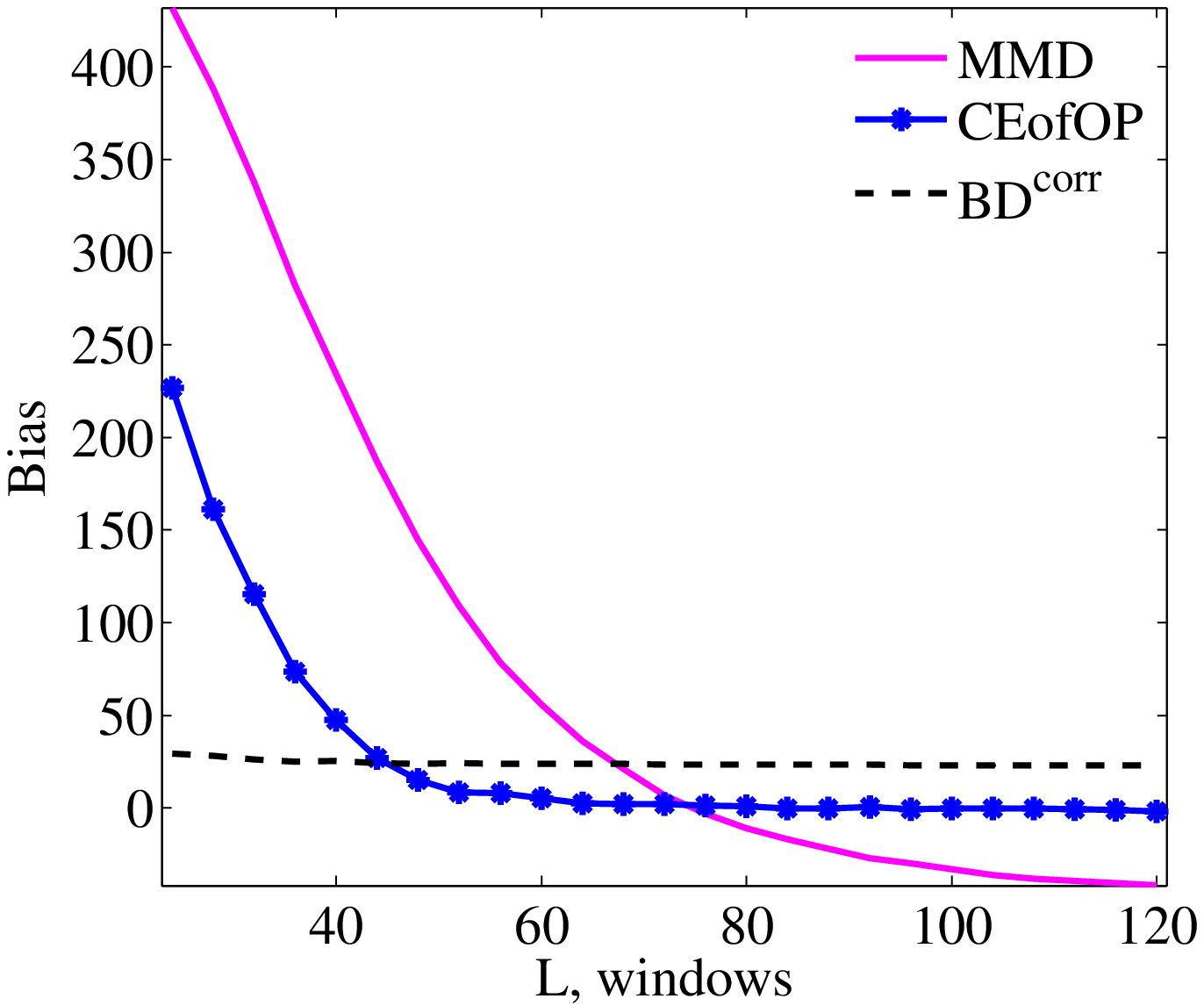}
	
\hspace{5mm}(d)
\end{minipage}
\caption{Measures of change-point detection performance for NL (a,~b) and AR (c,~d) processes with different lengths,
where $L$ is the product of window numbers given on the abscissa axis with window length $W=256$} 
\label{CPproblem1sample_fig}
\end{figure}

Summarizing,
performance of the CEofOP statistic strongly depends on the length of time series but is generally better than for the CMMD statistic.
In comparison with the classical Brodsky-Darkhovsky method, CEofOP has better performance for NL processes (see Figures~\ref{CPproblem1sample_fig}a,b),
and lower bias for AR processes (see Figure~\ref{CPproblem1sample_fig}d).

\subsection{Detecting multiple change-points}\label{sec4_ExpMultipleCP}

Here we investigate how well the considered statistics detect multiple change-points.
Methods for change-point detection via the CEofOP and the CMMD statistics are implemented according to Subsection~\ref{subsecCPdetectionMethod} and \cite[Subsection~4.5.1]{Unakafov2015}, respectively.
The Brodsky-Darkhovsky method is implemented according to \cite{BrodskyDarkhovskyKaplanShishkin1999} with only one exception:
to compute a threshold for it we use bootstrapping, which in our case provided better results than the technique described in \cite{BrodskyDarkhovskyKaplanShishkin1999}.
Nominal probability of a false alarm $\alpha = 0.05$ has been taken for all methods
(in the case of the CMMD statistic we have used the equivalent value $0.001$, see \cite[Subsection~4.3.2]{Unakafov2015}).

We consider here two processes, $\mathrm{AR}\big((0.3, 0.5, 0.1, 0.4), (t^\ast_1, t^\ast_2, t^\ast_3)\big)$ and \linebreak
$\mathrm{NL}\big((3.98, 4, 3.95, 3.8), (0.2, 0.2, 0.2, 0.3), (t^\ast_1, t^\ast_2, t^\ast_3)\big)$,
with change-points $t^\ast_k$ being independent and uniformly distributed in $\mleft\{\overline{t_k^\ast} - W, \overline{t_k^\ast} - W + 1, \ldots, \overline{t_k^\ast} + W\mright\}$
for $k = 1,2,3$ with $\overline{t_1^\ast} = 0.3L$, $\overline{t_2^\ast} = 0.7L$, $\overline{t_3^\ast} = 0.9L$, and $L = 100 \,W$.
For both processes we generate $N = 10000$ realizations $x^j$ with $j = 1, \ldots, N$.
We consider unequal lengths of stationary segments to study methods for change-point detection in more realistic conditions.

As we apply change-point detection via a statistic $S$ to realization $x^j$,
we obtain estimates of the number $\widehat{N}_\text{st}(S; x^j)$ of stationary segments
and of change-points positions $\widehat{t}_{l}^\ast(S; x^j)$ for $l = 1,2,\ldots,\widehat{N}_\text{st}(S; x^j)-1$.
Since the number of estimated change-points may be different from the actual number of changes,
we suppose that the estimate for $t_k^\ast$ is provided by the nearest $\widehat{t}_{l}^\ast(S; x^j)$.
Therefore the error of estimation of the $k$-th change-point provided by $S$ is given by
\begin{equation*}
	\mathrm{err}^j_k(S, X) = \min_{l = 1,2,\ldots,\widehat{N}_\text{st}(S; x^j)-1}  \Big| \widehat{t}_{l}^\ast(S; x^j) - t_k^\ast \Big|.
\end{equation*}
To assess the overall accuracy of change-point detection, we compute two quantities.
The fraction $\mathrm{sE}_k$ of satisfactory estimates of a change-point $t^\ast_k$, $k = 1,2,3$ is given by
\begin{equation*}
\mathrm{sE}_k(S, X) = \frac{\#\big\{j \in \{1,2,\ldots,N\} \mid \mathrm{err}^j_k(S, X) \leq \mathrm{MaxErr} \big\}}{N},
\end{equation*}
where $\mathrm{MaxErr}$ is the maximal satisfactory error; we take $\mathrm{MaxErr} = W = 256$.
The average number of {\it false change-points} is defined by:
\begin{equation*}
\mathrm{fCP}(S, X) = \frac{\sum\limits_{j=1}^{N}\!\big( \widehat{N}_\text{st}(S; x^j)-1-\#\big\{k \in \{1,2,3\} \mid \mathrm{err}^j_k(S, X) \leq \mathrm{MaxErr} \big\} \big)}{N}.
\end{equation*}

Results of the experiment are presented in Tables~\ref{CPproblem3_NL_tbl} and~\ref{CPproblem3_AR_tbl}, the best values are shown in {\bf bold}.

\begin{table}[!th]
\centering
\begin{tabular}{ l | c | c | c c c c }
\toprule
Statistic             &fCP       &\multicolumn{4}{c}{Fraction $\mathrm{sE}_k$ of satisfactory estimates} \\
&          &1st change   &2nd change   &3rd change   &average   \\
\midrule
$\mathrm{cMMD}$\,     &     1.17 &     0.465   &     0.642   &     0.747   &     0.618   \\
$\mathrm{CEofOP}$\,   &{\bf 0.62}&{\bf 0.753}  &{\bf 0.882}  &{\bf 0.930}  &{\bf 0.855}  \\
$\mathrm{BD^{corr}}$\,&     1.34 &     0.296   &     0.737   &     0.751   &     0.595   \\
\bottomrule
\end{tabular}\vspace{2mm}
\caption{Performance of change-point detection methods for the process with three change-points $\mathrm{NL}\big((3.98, 4, 3.95, 3.8), (0.2, 0.2, 0.2, 0.3), (t^\ast_1, t^\ast_2, t^\ast_3)\big)$}
\label{CPproblem3_NL_tbl}
\end{table}

\begin{table}[!th]
\centering
\begin{tabular}{ l | c | c | c c c c }
\toprule
Statistic             &fCP       &\multicolumn{4}{c}{Fraction $\mathrm{sE}_k$ of satisfactory estimates} \\
&          &1st change   &2nd change   &3rd change   &average   \\
\midrule
$\mathrm{CMMD}$\,     &     1.17 &     0.340   &     0.640   &     0.334   &     0.438\\
$\mathrm{CEofOP}$\,   &     1.12 &     0.368   &     0.834   &     0.517   &     0.573\\
$\mathrm{BD^{corr}}$\,&{\bf 0.53}&{\bf 0.783}  &{\bf 0.970}  &{\bf 0.931}  &{\bf 0.895}\\
\bottomrule
\end{tabular}\vspace{2mm}
\caption{Performance of change-point detection methods for the process with three change-points $\mathrm{AR}\big((0.3, 0.5, 0.1, 0.4), (t^\ast_1, t^\ast_2, t^\ast_3)\big)$}
\label{CPproblem3_AR_tbl}
\end{table}

In summary,
since distributions of ordinal patterns for NL and AR processes have different properties, results for them differ significantly.
The CEofOP statistic provides good results for the NL processes.
However, for the AR processes its performance is much worse: only the most prominent change is detected rather well.
Weak results for two other change-points are caused by the fact that the CEofOP statistic is rather sensitive to the lengths of stationary segments
(we have already seen this in Subsection~\ref{sec4_ExpEstimationSample}), and in this case they are not very long.

\section{Conclusions and open points}\label{SecConcl}

In this paper we have introduced a method for change-point detection via the CEofOP statistic and have tested it for time series coming from two classes of models having quite different distributions, namely piecewise stationary noisy logistic and autoregressive processes.

The empirical investigations suggest that the method proposed provides better detection of ordinal change-points
than the ordinal-patterns-based method introduced in \cite{SinnGhodsiKeller2012, SinnKellerChen2013}.
Performance of our method for the two model classes considered is particularly comparable to that for the classical Brodsky-Darkhovsky method,
but in contrast to it, ordinal-patterns-based methods require less a priori knowledge about the time series.
This can be especially useful in the case of considering non-linear models where the autocorrelation function does not describe distributions completely. Here the point is that with exception of the mean much of the distribution is captured by its ordinal structure. So (together with methods finding changes in mean) the CEofOP statistic can be used at least for a first exploration step.

Although numerical experiments and tests to real-world data cannot replace rigorous theoretical studies, the results of the current study show the potential of the change-point detection via the CEofOP statistic. 
However, there are some open points listed below.
\begin{enumerate}
\item A method for computing a threshold $h$ for the $\CEofOP$ statistic without using bootstrapping is of interest, since the bootstrapping procedure is rather time consuming.
One possible solution is to utilize Theorem~\ref{CEofOP_change} (Section~\ref{sec_stationary}) and to precompute thresholds using the values of $\Delta^d_{\gamma,\theta}(P,Q)$.
However, this approach requires further investigation.
\item The binary segmentation procedure \cite{Vostrikova1981} is not the only possible method for detecting multiple change-points.
        In \cite{Lavielle1999, LavielleTeyssiere2007} an alternative approach is suggested:
        the number of stationary segments $\widehat{N}_\text{st}$ is estimated by optimizing a contrast function, then the positions of the change-points are adjusted.
        Likewise one can consider a method for multiple change-point detection based on maximizing the following 
generalization of $\CEofOP$ statistic:        
        \begin{multline*}
        \CEofOP(t) = (L - d \widehat{N}_\text{st}) \,\eCE\Big(\big(\pi(k)\big)_{k=d}^L\Big)\\
		- \sum_{l=1}^{\widehat{N}_\text{st}}\big(\widehat{t}^\ast_{l} - \widehat{t}^\ast_{l-1} - d\big) \,\eCE\Big(\pi\big(\widehat{t}^\ast_{l-1}+d\big), \ldots,\pi\big(\widehat{t}^\ast_{l}\big)\Big),
	\end{multline*}
        where $\widehat{N}_\text{st}\in \mathbb{N}$  is an estimate of number of stationary segments,
        $\widehat{t}^\ast_1=0$, $\widehat{t}^\ast_{\widehat{N}_\text{st}} = L$
        and $\widehat{t}^\ast_1, \widehat{t}^\ast_2, \ldots,\widehat{t}^\ast_{\widehat{N}_\text{st}-1} \in \mathbb{N}$ are estimates of change-points.
        Further investigation in this direction could be of interest.
  \item As we have seen in Subsection~\ref{sec4_ExpEstimationSample}, CEofOP statistic requires rather large sample sizes to provide reliable change-point detection. This is due to the necessity of the empirical conditional entropy estimation (see Subsection~\ref{subsecCPdetectionMethod}). In order to reduce the required sample size, one may consider more effective estimates of the conditional entropy, for instance, the Grassberger estimate (see \cite{Grassberger2003} and \cite[Subsection~3.4.1]{Unakafov2015}). However elaboration of this idea is beyond the scope of this paper.
  \item In this paper only one-dimensional time series are considered, though there is no principal limitation for applying ordinal-patterns-based methods to multivariate data (see \cite{Keller2012}). Discussion of using ordinal-patterns-based methods for detecting change-point in multivariate data (for instance, in multichannel EEG) is therefore of interest.
  \item We have considered here only the ``off-line'' detection of changes, which is used when the acquisition of a time series is completed.
        Meanwhile, in many applications it is necessary to detect change-points ``on-line'', based on a small number of observations after the change \cite{BassevilleNikiforov1993}.
        Development of on-line versions of ordinal-patterns-based methods for change-point detection may be an interesting direction of a future work.
\end{enumerate}

\section{Asymptotic behavior of the CEofOP statistic}\label{sec_stationary}

Here we consider the values of $\CEofOP$ for the case when segments of a stochastic process before and after the change-point have infinite length.

Let us first introduce some notation.
Given an ordinal-$d^+$-stationary stochastic process $X$ for $d\in\mathbb{N}$
the distribution of pairs of ordinal patterns is denoted by $P = (p_{i,j})_{i,j\in S_d}$,
with $p_{i,j} = \mathbb{P}\big(\Pi(t) = i,\Pi(t+1) = j\big) = p_{j|i} p_{i}$ for all $i,j \in S_d$.
One easily sees the following:
The conditional entropy of ordinal patterns is represented as $\CE(X, d) = H(P)$, where
\begin{equation*}
H(P) = -\sum\limits_{i\in S_d} \sum\limits_{j\in S_d}p_{i,j} \ln p_{i,j} + \sum\limits_{i\in S_d}\bigg(\sum\limits_{j\in S_d}p_{i,j}\bigg)\ln\sum\limits_{j\in S_d}p_{i,j}.
\end{equation*}
Here recall that $p_i=\sum\limits_{j\in S_d}p_{i,j}$.

\begin{theorem}\label{CEofOP_change}
Let $Y=(Y_t)_{t\in {\mathbb N}}$ and $Z=(Z_t)_{t\in {\mathbb N}}$ be ergodic $d^+$-ordinal-statio\-nary stochastic processes on a probability space $(\Omega, \mathcal{A}, \mathbb{P})$
with probabilities of pairs of ordinal patterns of order $d\in \mathbb{N}$ given by
$P = (p_{i,j})_{i,j\in S_d}$ and $Q = (q_{i,j})_{i,j\in S_d}$, respectively.
For $L\in \mathbb{N}$ and $\gamma \in (0,1)$, let $\Pi_{L,\gamma}$ be the abstract sequence of ordinal patterns of order $d$ of
\begin{equation}\label{seqq}
(Y_1,\ldots ,Y_{\lfloor \gamma L \rfloor},Z_{\lfloor \gamma L \rfloor +1},Z_{\lfloor \gamma L \rfloor +2},\ldots ,Z_L).
\end{equation}
Then, for all $\theta \in (0,1)$ it holds
\begin{equation*}
\lim_{L \to \infty}\CEofOP\Big(\lfloor \theta L \rfloor; \Pi_{L,\gamma}\Big) = L\, \Delta^d_{\gamma,\theta}(P,Q)
\end{equation*}
$\mathbb{P}$-almost sure, where
\begin{equation*}
\Delta^d_{\gamma, \theta}(P,Q)\!=\!\begin{cases}
\!H\big(\gamma P + (1\!-\!\gamma)Q\big) - \theta H(P) - (1\!-\!\theta) H\big(\frac{\gamma - \theta}{1 - \theta}P + \frac{1 - \gamma}{1 - \theta}Q\big),&\!\!\theta < \gamma\\
\!H\big(\gamma P + (1\!-\!\gamma)Q\big) - \theta H\big(\frac{\gamma}{\theta}P + \frac{\theta-\gamma}{\theta}Q\big) - (1\!-\!\theta)H(Q),               &\!\!\theta \geq \gamma
\end{cases}.
\end{equation*}
\end{theorem}
By definition, \eqref{seqq} is a stochastic process of length $L+1$
with a potential ordinal change-point $t^\ast=\lfloor \theta L \rfloor$, i.e.~the position of $t^\ast$ relative to $L$ is principally the same for all $L$, and the statistics considered are stabilizing for increasing $L$. \eqref{seqq} can be particularly interpreted as a part of a stochastic process including exactly one ordinal chance point.
We omit the proof of Theorem \ref{CEofOP_change} since it is a simple computation.

Due to the properties of the conditional entropy, it holds
\begin{equation*}
\max_{\theta \in (0,1)} \Delta^d_{\gamma,\theta}(P,Q) = \Delta^d_{\gamma,\gamma}(P,Q) = H\big(\gamma P + (1-\gamma)Q\big) - \gamma H(P) - (1 - \gamma)H(Q).
\end{equation*}

Values of $\Delta^d_{\gamma,\theta}(P,Q)$ can be computed for a piecewise stationary stochastic process
with known probabilities of ordinal patterns before and after the change-point\footnote{To apply Theorem~\ref{CEofOP_change}
one needs probabilities of pairs of ordinal patterns of order $d$.
They can be calculated from the probabilities of ordinal patterns of order $(d+1)$:
as one can easily verify, probability $P_{i,j}$ of any pair $(i, j)$ of ordinal patterns $i,j \in S_d$
is equal either to the probability of a certain ordinal pattern of order $(d+1)$ or to the sum of two such probabilities.}.
In \cite{BandtShiha2007} authors compute probabilities of ordinal patterns of orders $d = 2$ (Proposition~5.3) and $d = 3$ (Theorem~5.5)
for stationary Gaussian processes (in particular, for autoregressive processes).
Here we use these results to provide Example~\ref{CEofOPARprocess} illustrating Theorem~\ref{CEofOP_change}.

\begin{example}\label{CEofOPARprocess}
Consider an autoregressive process $\AR\big((\phi_1, \phi_2), t^\ast\big)$ with a single change-point $t^\ast = L/2$ for $L \in \mathbb{N}$.
Using the results from \cite{BandtShiha2007} we compute distributions $P^{\phi_1}$, $P^{\phi_2}$ of ordinal pattern pairs for orders $d = 1,\,2$ and on this basis
we calculate the values of $\Delta^d_{0.5, 0.5}\mleft( P^{\phi_1}, P^{\phi_2} \mright)$ for different values of $\phi_1$ and $\phi_2$.
The results are presented in Tables~\ref{CEofOPforAR_d1_tbl} and~\ref{CEofOPforAR_d2_tbl}.
\begin{table}[!th]
\centering
\begin{tabular}{ l | l l l l l l l l l l l}
\toprule
\backslashbox{$\phi_2$}{$\phi_1$} &0.00& 0.10& 0.20& 0.30& 0.40& 0.50& 0.60& 0.70& 0.80& 0.90& 0.99\\
\midrule
0.00			      &0   & 0.02& 0.07& 0.15& 0.26& 0.40& 0.56& 0.74& 0.95& 1.18& 1.44\\
0.10			      &0.02& 0   & 0.02& 0.06& 0.14& 0.25& 0.37& 0.53& 0.71& 0.91& 1.13\\
0.20			      &0.07& 0.02& 0   & 0.02& 0.06& 0.13& 0.23& 0.36& 0.51& 0.68& 0.88\\
0.30			      &0.15& 0.06& 0.02& 0   & 0.01& 0.06& 0.13& 0.22& 0.34& 0.49& 0.66\\
0.40			      &0.26& 0.14& 0.06& 0.01& 0   & 0.01& 0.06& 0.12& 0.22& 0.33& 0.48\\
0.50			      &0.40& 0.25& 0.13& 0.06& 0.01& 0   & 0.01& 0.05& 0.12& 0.21& 0.33\\
0.60			      &0.56& 0.37& 0.23& 0.13& 0.06& 0.01& 0   & 0.01& 0.05& 0.12& 0.21\\
0.70			      &0.74& 0.53& 0.36& 0.22& 0.12& 0.05& 0.01& 0   & 0.01& 0.05& 0.12\\
0.80			      &0.95& 0.71& 0.51& 0.34& 0.22& 0.12& 0.05& 0.01& 0   & 0.01& 0.05\\
0.90			      &1.18& 0.91& 0.68& 0.49& 0.33& 0.21& 0.12& 0.05& 0.01& 0   & 0.01\\
0.99			      &1.44& 1.13& 0.88& 0.66& 0.48& 0.33& 0.21& 0.12& 0.05& 0.01& 0   \\
\bottomrule
\end{tabular}\vspace{2mm}
\caption{Values of $100\,\Delta^1_{0.5, 0.5}\mleft( P^{\phi_1}, P^{\phi_2} \mright)$ for an autoregressive process (coefficient $100$ here is only for the sake of readability)}
\label{CEofOPforAR_d1_tbl}
\end{table}

\begin{table}[!th]
\centering
\begin{tabular}{ l | l l l l l l l l l l l}
\toprule
\backslashbox{$\phi_2$}{$\phi_1$} &0.00& 0.10& 0.20& 0.30& 0.40& 0.50& 0.60& 0.70& 0.80& 0.90& 0.99\\
\midrule
0.00			      &0   & 0.04& 0.15& 0.33& 0.56& 0.85& 1.18& 1.55& 1.95& 2.40& 2.88\\
0.10			      &0.04& 0   & 0.04& 0.14& 0.31& 0.53& 0.80& 1.12& 1.48& 1.89& 2.34\\
0.20			      &0.15& 0.04& 0   & 0.03& 0.13& 0.29& 0.51& 0.77& 1.08& 1.44& 1.85\\
0.30			      &0.33& 0.14& 0.03& 0   & 0.03& 0.13& 0.28& 0.49& 0.75& 1.06& 1.43\\
0.40			      &0.56& 0.31& 0.13& 0.03& 0   & 0.03& 0.12& 0.27& 0.48& 0.74& 1.06\\
0.50			      &0.85& 0.53& 0.29& 0.13& 0.03& 0   & 0.03& 0.12& 0.27& 0.48& 0.74\\
0.60			      &1.18& 0.80& 0.51& 0.28& 0.12& 0.03& 0   & 0.03& 0.12& 0.27& 0.48\\
0.70			      &1.55& 1.12& 0.77& 0.49& 0.27& 0.12& 0.03& 0   & 0.03& 0.12& 0.28\\
0.80			      &1.95& 1.48& 1.08& 0.75& 0.48& 0.27& 0.12& 0.03& 0   & 0.03& 0.13\\
0.90			      &2.40& 1.89& 1.44& 1.06& 0.74& 0.48& 0.27& 0.12& 0.03& 0   & 0.03\\
0.99			      &2.88& 2.34& 1.85& 1.43& 1.06& 0.74& 0.48& 0.28& 0.13& 0.03& 0   \\

\bottomrule\vspace{2mm}
\end{tabular}

\caption{Values of $100\,\Delta^2_{0.5, 0.5}\mleft( P^{\phi_1}, P^{\phi_2} \mright)$ for an autoregressive process}
\label{CEofOPforAR_d2_tbl}
\end{table}

According to Theorem~\ref{CEofOP_change}, for $\pi^{d,L} = \big(\pi(k)\big)_{k=d}^L$ being a sequence of ordinal patterns of order $d$ for a realization of
$\AR\big((\phi_1, \phi_2), L/2\big)$ it holds almost sure that
\begin{equation*}
\frac{1}{L}\max_{\theta \in (0,1)} \CEofOP\big(\lfloor \theta L \rfloor; \pi^{d,L}\big) \xrightarrow[L \to \infty]{} \Delta^d_{0.5, 0.5}\mleft( P^{\phi_1}, P^{\phi_2} \mright).
\end{equation*}
Figure~\ref{CEofOPforAR_conv_fig} shows how fast this convergence is.
Note that the $\CEofOP$ statistic for orders $d=1,2$ allows to distinguish between change and no change in the considered processes for $L \geq 20\cdot10^3$.
For $L = 10^5$ the values of the $\CEofOP$ statistic for order $d=2$ is already very close to its theoretical values, whereas for $d=1$ this length does not seem sufficient.

\begin{figure}[!th]
\centering
\begin{minipage}[h]{0.49\hsize}
\centering
\includegraphics[scale=0.41]{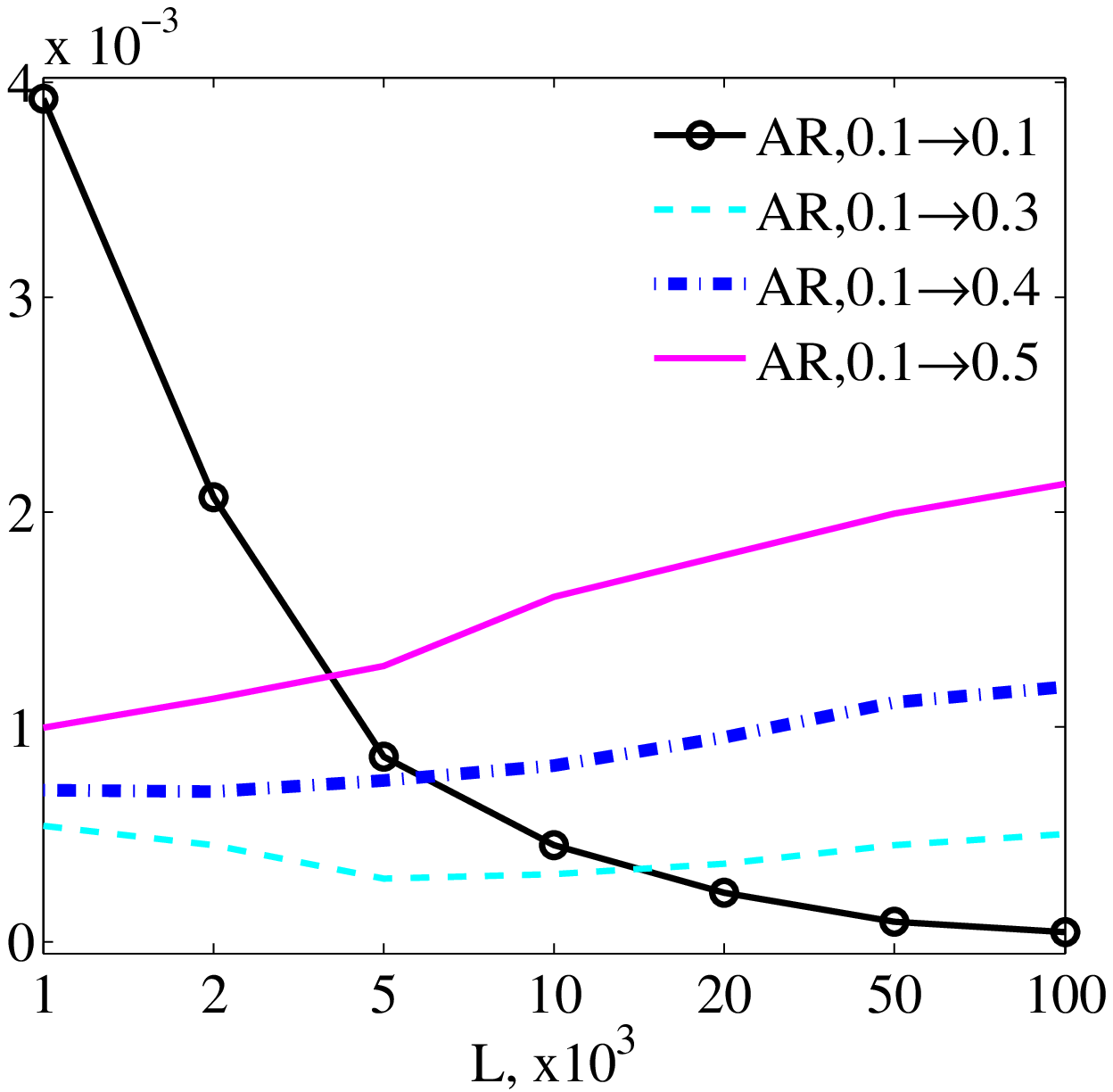}
	
(a)
\end{minipage}
\begin{minipage}[h]{0.49\hsize}
\centering
\includegraphics[scale=0.41]{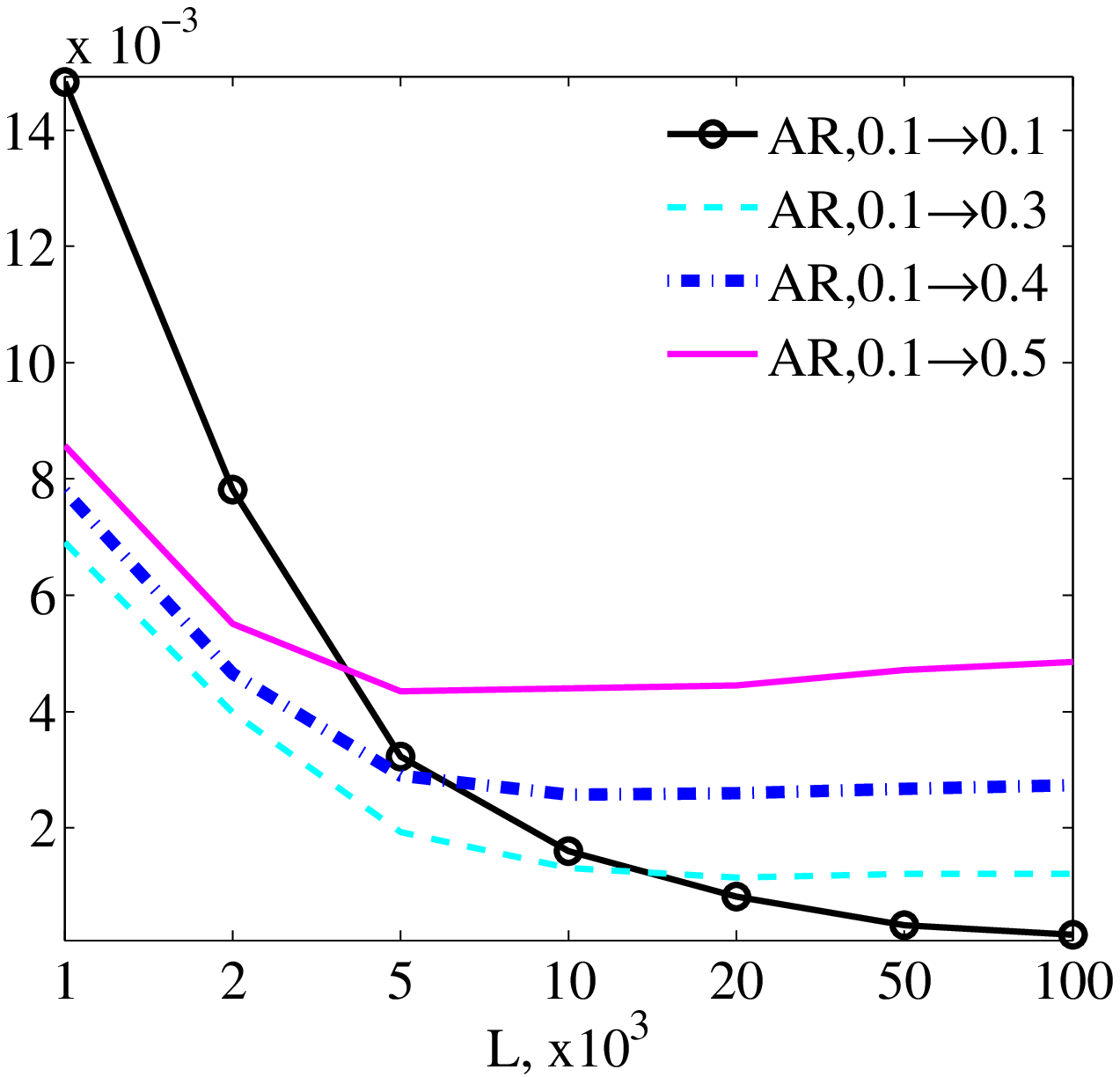}
	
(b)
\end{minipage}

\caption{Convergence of $\frac{1}{L}\max\limits_{\theta \in (0,1)} \CEofOP\big(\lfloor \theta L \rfloor; \pi^{d,L}\big)$
to the theoretical values $\Delta^d_{0.5, 0.5}$ as $L$ increases for autoregressive processes for $d=1$ (a) and $d=2$ (b).
Here AR, $\phi_1\rightarrow \phi_2$ stands for the process $\AR\big((\phi_1, \phi_2), L/2\big)$.
The provided empirical values are obtained either as 5th percentile (for $\phi_1 \neq \phi_2$) or as 95th percentile (for $\phi_1 = \phi_2$) from 1000 trials}
\label{CEofOPforAR_conv_fig}
\end{figure}
\end{example}

The following result is a simple consequence of Theorem~\ref{CEofOP_change}.
\begin{corollary}\label{CEofOP_nochange}
Let $X=(X_t)_{t\in {\mathbb N}_0}$ be an ergodic $d^+$-ordinal-statio\-nary stochastic process on a probability space $(\Omega, \mathcal{A}, \mathbb{P})$.
For $L\in \mathbb{N}$ let $\Pi^{d,L}$ be the abstract sequence of ordinal patterns of order $d$ of
$(X_0,X_1,\ldots ,X_L)$.
Then for any $\theta \in (0,1)$ it holds
\begin{equation}\label{CEofOP_nochange_eq}
\lim_{L \to \infty}\CEofOP\Big(\lfloor \theta L \rfloor; \Pi^{d,L} \Big) = 0
\end{equation}
$\mathbb{P}$-almost sure.
\end{corollary}

\section*{Acknowledgement}

This work was supported by the Graduate School for Computing in Medicine and Life Sciences
funded by Germany's Excellence Initiative [DFG GSC 235/1].

\ifArxiv	
\bibliographystyle{plain}
\else		
\bibliographystyle{gNST}
\fi
\bibliography{UK2015}

\end{document}